\documentclass[11pt]{amsart}
\usepackage{a4wide}
\usepackage{amssymb,extarrows}
\usepackage{amsthm}
\usepackage{amsmath}
\usepackage{amscd}
\usepackage{xcolor}
\usepackage[mathscr]{eucal}
\usepackage[all]{xy}
\usepackage{bbm}
\usepackage{verbatim}

\numberwithin{equation}{section}

\theoremstyle{plain}
\newtheorem{theorem}{Theorem}[section]
\newtheorem{corollary}[theorem]{Corollary}
\newtheorem{lemma}[theorem]{Lemma}
\newtheorem{proposition}[theorem]{Proposition}
\theoremstyle{definition}
\newtheorem{definition}[theorem]{Definition}
\newtheorem{remark}[theorem]{Remark}

\theoremstyle{remark}

\newcommand{\A}{\mathbb{A}}
\newcommand{\R}{\mathbb{R}}
\newcommand{\Q}{\mathbb{Q}}
\newcommand{\Z}{\mathbb{Z}}
\newcommand{\N}{\mathbb{N}}
\newcommand{\C}{\mathbb{C}}

\renewcommand{\H}{\mathbb{H}}

\newcommand{\G}{\mathbb{G}}

\newcommand{\zxz}[4]{\begin{pmatrix} #1 & #2 \\ #3 & #4 \end{pmatrix}}

\newcommand{\leg}[2]{\left( \frac{#1}{#2} \right)}
\newcommand{\kzxz}[4]{\left(\begin{smallmatrix} #1 & #2 \\ #3 & #4\end{smallmatrix}\right) }

\newcommand{\im}{\operatorname{Im}}
\newcommand{\re}{\operatorname{Re}}

\newcommand{\calA}{\mathcal{A}}
\newcommand{\calB}{\mathcal{B}}
\newcommand{\calC}{\mathcal{C}}
\newcommand{\calD}{\mathcal{D}}

\newcommand{\calH}{\mathcal{H}}

\newcommand{\calK}{\mathcal{K}}

\newcommand{\calM}{\mathcal{M}}

\newcommand{\calQ}{\mathcal{Q}}

\newcommand{\calS}{\mathcal{S}}
\newcommand{\calT}{\mathcal{T}}

\newcommand{\calZ}{\mathcal{Z}}

\newcommand{\frake}{\mathfrak e}
\newcommand{\frakg}{\mathfrak g}
\newcommand{\frakp}{\mathfrak p}

\newcommand{\frakt}{\mathfrak t}

\newcommand{\bs}{\backslash}

\newcommand{\vol}{\operatorname{vol}}

\newcommand{\KM}{\operatorname{KM}}
\newcommand{\reg}{\operatorname{reg}}

\newcommand{\Span}{\operatorname{span}}

\newcommand{\SL}{\operatorname{SL}}
\newcommand{\GL}{\operatorname{GL}}
\newcommand{\Sp}{\operatorname{Sp}}

\newcommand{\GSpin}{\operatorname{GSpin}}

\newcommand{\Sym}{\operatorname{Sym}}

\newcommand{\Aut}{\operatorname{Aut}}

\newcommand{\SO}{\operatorname{SO}}

\newcommand{\supp}{\operatorname{supp}}
\newcommand{\sig}{\operatorname{sig}}

\newcommand{\Div}{\operatorname{Div}}

\newcommand{\End}{\operatorname{End}}

\newcommand{\ord}{\operatorname{ord}}
\newcommand{\id}{\operatorname{id}}

\newcommand{\pr}{\operatorname{pr}}

\begin{document}

\title[A converse theorem for Borcherds products]{A converse theorem for Borcherds products and the injectivity of the Kudla-Millson theta lift}
\author{Oliver Stein}
\address{Fakult\"at f\"ur Informatik und Mathematik\\ Ostbayerische Technische Hochschule Regensburg\\Galgenbergstrasse 32\\93053 Regensburg\\Germany}
\email{oliver.stein@oth-regensburg.de}

\begin{abstract}
  We prove a converse theorem for the multiplicative Borcherds lift for lattices of square-free level whose associated  discriminant group is anisotropic. This can be seen as generalization of Bruinier's results in \cite{Br2}, which provides a converse theorem for lattices of prime level. The surjectivity of the Borcherds lift in our case follows from the injectivity of the Kudla-Millson theta lift. We generalize the corresponding results in \cite{BF1} to the aforementioned lattices and thereby in particular to lattices which are not unimodular and not of type $(p,2)$. Along the way, we compute the contribution of both, the non-Archimedean and Archimedean places of the $L^2$-norm of the Kudla-Millson theta lift.  As an application we refine a theorem of Scheithauer on the non-existence of reflective automorphic products. 
\end{abstract}

\maketitle

\section{Introduction}
In his celebrated paper \cite{B}, Borcherds constructed a multiplicative lifting (referred to as {\it Borcherds lift}) from the space of vector valued weakly holomorphic modular forms of weight $1-\frac{p}{2}$ for the Weil representation associated to a lattice $L$ of type $(p,2)$ and rank $m\in 2\Z$ to meromorphic modular forms for the orthogonal group of $L$. The orthogonal modular forms arising this way are of special interest as they allow an infinite product expansion at each cusp and have a divisor being a linear combination of so-called Heegner divisors. Therefore, the following question, initially raised by Borcherds (see \cite{B}, Problem 16.10), is of great importance:
Let $F$ be a meromorphic modular form for the orthogonal group of $L$ whose divisor is a linear combination of Heegner divisors. Is there a weakly holomorphic modular form of weight $1-\frac{p}{2}$ for the Weil representation  attached to the lattice $L$ whose Borcherds lift is (up to a multiplicative constant) the form $F$? This question has been addressed in several papers: It is pointed out in \cite{Br2}, that in the case of lattices with signature $(1,2)$ there are orthogonal modular forms which cannot be obtained as a Borcherds lift of some vector valued modular form. However, an affirmative answer for a large class of lattices is given in \cite{Br1} and \cite{Br2}.  The most general results in this direction can be found  in \cite{Br2}:
\begin{enumerate}
  \item[i)]
    In Theorem 1.2 a converse theorem is given under the assumption that the lattice $L$ allows a decomposition over $\Z$ of the form
\begin{equation}\label{eq:lattice_decomp}
  L = M\oplus U(N)\oplus U,
\end{equation}
 where $M$ is a lattice of type $(p-2,0)$, $U$ is a hyperbolic plane and $U(N)$ is a scaled hyperbolic plane, i. e. the hyperbolic plane $U$ equipped with the quadratic form $Q_N((x_1,x_2))=Nx_1x_2$. 
\item[ii)]
  The Theorem 1.4 also states a converse theorem. But it does not rely on the assumption that the lattice splits a hyperbolic plane. It only requires that the level of the lattice is a prime number. 
  \end{enumerate}
In both of the above theorems it is presumed that $p \ge 3$.

In \cite{BF}, a converse theorem under the hypothesis that $L$ is {\it unimodular} is proved. The main portion of the proof shows that the Kudla-Millson theta lift, a map from the space of cusp forms of weight $\frac{m}{2}$  transforming with the Weil representation into the space of closed differential 2-forms on the modular variety $X$, is injective (see Theorem 4.9).  This is achieved by evaluating the $L^2$-norm of the Kudla-Millson theta lift, showing that it is non-zero for most choices of $m$ and the Witt rank $r_0$ of $L$. The important paper \cite{BF1} establishes a relation between the regularized theta lift (see \eqref{eq:regularized_theta_lift}) and the Kudla-Millson theta lift (see \cite{BF1}, Theorem 6.1). Based on this relation and a weak converse theorem in \cite{Br1}, Theorem 4.23, the surjectivity of the Borcherds lift is derived.

The purpose of this paper is threefold.
\begin{enumerate}
\item[i)]
We give a further converse theorem for the Borcherds lift. This theorem may be seen as a generalization of  Bruinier's Theorem 1.4 in \cite{Br2}. Our converse theorem does not rely on the decomposition \eqref{eq:lattice_decomp}. However, it assumes that the level of the underlying lattice $L$ is square-free and the associated discriminant group $L'/L$ is anisotropic. This condition means that each $p$-group of $L'/L$ is of the form \eqref{eq:anisotropic_p_groups}. Although we have a restriction on the structure of the $p$-groups of $L'/L$, our results may be interpreted as generalisation of Theorem 1.4 in \cite{Br2} to lattices with square-free level (to the best of my knowledge such a result has not yet been established). 
\item[ii)]
As a byproduct, but probably interesting in its own right,  we show the injectivity of the Kudla-Millson theta lift associated to a lattice which is not unimodular and not of type $(p,2)$. This generalizes the results in \cite{Br1} and \cite{Zu}. 
\item[iii)]
  As an application we refine a result of Scheithauer (Theorem 12.3 in \cite{Sch}), which states that there are only finitely many reflective and symmetric automorphic products of singular weight.

  As far as I am aware, most of the recent papers on the classification of reflective modular forms are based  on Bruinier's converse theorems in \cite{Br2} and thereby rely either on the assumption that the involved lattice splits $U\oplus U(N)$ over $\Z$ or has prime number level. It is conceivable that these assumptions can be weakened by employing the converse theorem of the present paper. I hope to come to back to these topics in the future. 
\end{enumerate}

Let us explain the content of the paper in some more detail. 
Let $(L,(\cdot,\cdot))$ be a non-degenerate even lattice of type $(p,2)$ equipped  with a bilinear form $(\cdot,\cdot)$  and the associated quadratic form $x\mapsto Q(x)=\frac{1}{2}(x,x)$. We denote the rank of $L$ with $m$ and assume in the whole paper that $m$ is even.  Moreover, let $V(\Q)$ be the rational  quadratic space $L\otimes \Q$,  
$L'$ the dual lattice of $L$ and $A=L'/L$ be the associated discriminant group. Since $L$ is even, the quadratic form $Q$ induces a $\Q/\Z$-valued quadratic form on $A$, which thereby becomes also a quadratic module.
The (finite) Weil representation $\rho_A$ is a unitary representation of $\SL_2(\Z)$ on the group ring $\C[A]$,
\[
\rho_A:\SL_2(\Z)\longrightarrow \GL(\C[A]).
\]
We denote the standard basis of $\C[A]$ by $(\frake_\mu)_{\mu\in A}$.
A weakly holomorphic modular form of weight $\kappa\in \Z$ for $\SL_2(\Z)$ of type $\rho_A$ is a holomorphic function $f$ on the upper half plane $\H$ which satisfies $f(\gamma\tau) = (c\tau+d)^\kappa\rho_A(\gamma)f(\tau)$ for all $\gamma=\kzxz{a}{b}{c}{d}\in\SL_2(\Z)$. 
Additionally, $f$ is meromorphic at the cusp $\infty$ (see Section \ref{subsec:weak_maas_forms} for further details).
We denote the space of these forms by $M_{\kappa,A}^!$. A modular form of this type is called holomorphic if it is holomorphic instead of meromorphic at $\infty$. We write $M_{\kappa,A}$ for the space of all such  modular forms and use the symbol $S_{\kappa,A}$ for the subspace of cusp forms. 
Finally, let $V(\R) = L\otimes \R$ and $D$ the Grassmannian of 2-dimensional negative definite subspaces in $V(\R)$.

A vital step in the proof of the surjectivity of the Borcherds lift  is the proof of the injectivity of the Kudla-Millson theta lift $\Lambda$. It maps a cusp form $f\in S_{\kappa,A}$ of weight $\kappa = \frac{m}{2}$  to a closed differential $2$-form on $D$. More precisely, we have
\begin{align*}
  \Lambda: S_{\kappa,A}\longrightarrow \calZ^2(D),\quad f\mapsto \Lambda(f)
\end{align*}
with
\[
\Lambda(f)(z) = \int_{\SL_2(\Z)\bs \H}\langle f(\tau),\Theta_A(\tau,z,\varphi_{\KM})\rangle \frac{d u d v}{v^2}
\]
with $\tau = u+iv\in \H$. 
Here $\langle\cdot,\cdot\rangle$ is the standard scalar product on the group ring,
 \begin{equation}\label{eq:schwarz_form_tensor_prod}
\varphi_{\KM}\in\left[S(V)\otimes \calA^2(D)\right]^{O(V(\R))},
 \end{equation} is a Schwartz {\it form} constructed by Kudla and Millson in \cite{KM1} and $\Theta_A$ is a theta series associated to $\varphi_{\KM}$, where $\calA^2(D)$ means the space of smooth differential 2-forms on $D$ and $\calZ^2(D)$ the subspace of closed $2$-forms. See Section \ref{subsec:schwarz_forms} for the definition of the Kudla-Millson form and Section \ref{sec:inj_kudla_millson} for the definition of $\Theta_A$. In \cite{BF} the injectivity of $\Lambda$ is proved in a more general setting:
 \begin{enumerate}
 \item[i)]
   Instead of $\varphi_{\KM}$ a more generalized Schwartz form $\varphi_{q,l}$ due to \cite{FM} is considered. 
 \item[ii)]
   An adelic set-up for the involved Eisenstein series and theta series is utilized.
 \item[iii)]
   Instead of working with the Grassmannian $D$ the authors work with a Shimura variety $X_K$ of orthogonal type, whose complex points are given by
   \[
   X_K(\C) = H(\Q)\bs (D\times H(\A_f))/K,
   \]
   where $H = \GSpin(V)$, $\A_f$ the finite adeles of $\Q$ and $K\subset H(\A_f)$ is a compact open subgroup which leaves $L$ stable and acts trivially on $A$. 
 \end{enumerate}
 However, the proof is given under the assumption that $L$ is {\it unimodular}. Very recently, Zuffetti proved in \cite{Zu} the injectivity of $\Lambda$ for {\it non-unimodular} lattices but of signature $(p,2)$.  
 In the present paper, we prove the injectivity of $\Lambda$ in the same general setting, but we drop the aforementioned restrictions on $L$  and assume only that the associated discriminant group $A$ is anisotropic.  To this end, we produce for all relevant results in Section 4 of \cite{BF} a corresponding result in the vector valued setting. This includes a vector valued version of the Siegel-Weil formula, which can be easily deduced from the classical Siegel-Weil formula in \cite{KR1} and \cite{KR2}. Apart from Proposition 4.7 and Theorem 4.9, the transfer to the vector valued approach is mostly straightforward. The main difference is that we have to work with a special family of Schwartz-Bruhat functions $\varphi_\mu\in S(V(\A_f)), \; \mu \in A,$ on the finite adeles $\A_f$ (see \eqref{eq:familiy}).
 It  takes a lot more effort to obtain the analogue of Proposition 4.7 and Theorem 4.9 in \cite{BF}. The former result requires a vector valued version of the classical doubling formula (see \cite{Bo}), which has been given in a separate paper by the author (see \cite{St2}). Theorem 4.9 relies on the existence of the standard $L$-function associated to a Hecke eigenform $f$ of all Hecke operators $T(n^2)$. Such an $L$-function has been defined in another paper by the author  and the usual fundamental properties have been proved therein (see \cite{St3}). Thus, the following theorem and consequently this paper can be seen as a result of several papers of the author (\cite{St1}-\cite{St3}).
 \begin{theorem}\label{thm:intro_injectivity_kudla_millson}[Cor. \ref{cor:injectivity_kudla_millson}]
  Let $m > \max(6, 2l-2, 3+r_0)$ and $s_0= (m-3)/2$. Further, let $p>1$ and $q+l$ even and assume that $A$ is anisotropic. Then the theta lift $\Lambda: S_{\kappa,A}\rightarrow \calZ^q(X_K,\widetilde{\Sym}^l(V))$ is injective. 
 \end{theorem}

 Based on  Theorem \ref{thm:intro_injectivity_kudla_millson}, subsequently the surjectivity of the Borcherds lift is established. The main result regarding the Borcherds lift is Theorem 13.3 of \cite{B}. It states that, given a weakly holomorphic modular form $g$ of weight $\kappa=1-\frac{p}{2}$ and type $\rho_A$ with Fourier coefficients $c(\mu,n)$ with  $\mu\in A$ and $n\in \Z+Q(\mu)$ (and $c(\mu,n)\in \Z$ for $n<0$), there exists a meromorphic modular form $\Psi_L$ for some subgroup of the orthogonal group of $L$ with
 \begin{enumerate}
 \item[i)]
   weight $\displaystyle \frac{c(0,0)}{2}$ and
 \item[ii)]
   a divisor given by
\[
   \sum_{\mu\in A}\sum_{\substack{n\in \Z+Q(\mu)\\n<0}}c(\mu,n)Z(n,\mu), 
   \]
   where $Z(\mu,n)$ is given by \eqref{eq:div_components}. 
 \end{enumerate}
 The modular form $\Psi_L$ is defined by means of the regularized theta lift
 \begin{equation}\label{eq:regularized_theta_lift}
 \Phi_L(g)(\tau,z) =  \int_{\SL_2(\Z)\bs \H}^{\reg}\langle g(\tau),\Theta_A(\tau,z,\varphi_\infty^{p,2})\rangle \frac{du dv}{v^2},
 \end{equation}
 where $\Theta_A$ is the theta series associated to the Gaussian $\varphi_\infty^{p,2}$ of signature $(p,2)$ and $\int_{\SL_2(\Z)\bs \H}^{\reg}$ is a regularization of the integral $\int_{\SL_2(\Z)\bs \H}$ according to \cite{B}. 
 The proof of our converse theorem makes use of the same approach as the one taken for Cor. 1.7 in \cite{BF}. It is based on Thm. 6.1 in \cite{BF1} and Thm. 4.23 in \cite{Br1}.  We stick with adelic setup we employed to establish the injectivity of the Kudla-Millson theta lift and generalize both of the aforementioned theorems to this setting. Based on these results we obtain
\begin{corollary}[Cor. \ref{cor:surjectivity_borcherds_lift}]
  Let $m$ be the even rank of the lattice $L$ satisfying $m>\max(6,3+r_0)$ and $m\equiv 0\bmod{4}$. Moreover, let the  associated discriminant form $A$ be anisotropic and $F:D^+\rightarrow \C$ be a meromorphic modular form of weight $r$ and character $\chi$ (of finite order) for the discriminant kernel $\Gamma(L)$  whose divisor is a linear combination of Heegner divisors. Then there exists a weakly holomorphic modular form $f\in M_{\ell,L^-}^!$ such that $F$ is up to a constant multiple the Borcherds lift $\Psi_L$ of $f$. 
\end{corollary}

\section{Preliminaries and notations}\label{sec:preliminaries}
In this section we fix some notation, which will be used this way throughout the paper unless it is stated otherwise, and recall some basic facts, which will be vital for the rest of the paper. For details the reader may consult \cite{BF} and \cite{BF1}. 

We start with a non-degenerate lattice $(L,(\cdot,\cdot))$ of type $(p,q)$ and even rank $m=p+q$. The signature of $L$ is defined by $\sig(L)= p-q$. Associated to the bilinear form $(\cdot,\cdot)$ we have the quadratic form $Q(\cdot) = \frac{1}{2}(\cdot,\cdot)$. We assume that $L$ is even, i. e. $Q(x) \in \Z$ for all $x\in L$.
Let 
\[
L':=\{x\in V=L\otimes \Q\; :\; (x,y)\in\Z\quad \text{ for all } \; y\in L\}
\]
be the dual lattice of the lattice even $L$. 
Since $L\subset L'$, the elementary divisor theorem implies that $L'/L$ is a finite group. We denote  this group by $A$. The modulo 1 reduction of both, the bilinear form $(\cdot, \cdot)$ and the associated quadratic form, defines a $\Q/\Z$-valued bilinear form $(\cdot,\cdot)$ with corresponding $\Q/2\Z$-valued quadratic form on $A$. We call $A$ combined with $(\cdot,\cdot)$ a discriminant group or a quadratic module. By $\sig(A) =\sig(L) = p-q$ we denote the signature of $A$.
We call $A$ anisotropic if $q(\mu) = 0$ holds only for $\mu=0$.
It is well known that any discriminant group can be decomposed into a direct sum of quadratic modules of the following form (cf. \cite{BEF})
\begin{equation}\label{eq:anisotropic_p_groups}
  \begin{split}
  &  \calA_{p^k}^t = \left(\Z/p^k\Z, \;\frac{tx^2}{p^k}\right),\; p> 2,\quad \calA_{2^k}^t=\left((\Z/2^k\Z,\; \frac{tx^2}{2^{k+1}}\right),\\
  & \calB_{2^k} = \left(\Z/2^k\Z\oplus\Z/2^k\Z;\; \frac{x^2+2xy+y^2}{2^k}\right),\quad \calC_{2^k} = \left(\Z/2^k\Z\oplus\Z/2^k\Z;\; \frac{xy}{2^k}\right).
    \end{split}
  \end{equation}
The structure of anisotropic finite quadratic modules is well known: In particular, for an odd prime $p$  each $p$-group $A_p$ of a discriminant form  $A$ can be either written as
\begin{equation}\label{eq:anisotropic_finite_modules}
  \calA_p^t \text{ or as a direct sum } \calA_p^t\oplus \calA_p^1.
\end{equation}
  For further details we refer to  \cite{BEF}. 

  We put $V=V(\Q)=L\otimes \Q$ and choose an orthogonal basis $\{v_i\}$ of $V(\R)=L\otimes \R$ such that $(v_\alpha, v_\alpha) = 1,\; \alpha = 1,\dots, p$ and $(v_\mu, v_\mu) = -1$ for $\mu=p+1,\dots,m$. For $x\in V$ we may write
\[
x = \sum_{\alpha = 1}^p x_\alpha v_\alpha + \sum_{\mu=p+1}^mx_\mu v_\mu
\]
in terms of its coordinates with respect to the basis $\{v_i\}$ such that
\begin{equation}\label{eq:quadratic_form_basis}
  (x,x) = \sum_{\alpha = 1}^p x_\alpha^2 - \sum_{\mu=p+1}^m x_\mu^2.
\end{equation}
By $r_0\in\Z$ we mean the {\it Witt index} of $V$, i. e. the dimension of a maximal isotropic subspace of $V$. 
Now pick the fixed subspace
\begin{equation}\label{eq:z_0}
  z_0=\Span\{v_\mu\; |\; \mu=p+1,\dots,m\}
  \end{equation}
and let
\begin{equation}\label{eq:D_+}
  D=\{z\subset V\;|\; \dim(z) = q, \; (\cdot,\cdot)_{|z} < 0\}
\end{equation}
the Grassmannian of oriented $q$- planes in $V(\R)$. It is well known that $D$ is a real analytic manifold. We denote by
\[
\calA^q(D) \text{ and } \calZ^q(D)
\]
the smooth differential $q$- forms and the smooth closed differential $q$- forms, respectively, on $D$.    
Note that $D$ has two connected components
\[
D = D^+\sqcup D^-
\]
given by the two possible choices of an orientation. 
Clearly, the orthogonal group $G(\R)=O(V(\R))$  acts on $D$. We denote by $K_\infty$ the subgroup of $G(\R)$ which stabilizes $z_0$. As $G(\R)$ acts transitively on $D$, we have  $G(\R)/K_\infty \cong D$. Note that $K_\infty\cong O(p)\times O(q)$.
Let $S(V(\R))$ be the space of Schwartz functions on $V(\R)$ and $G'(1)(\R) = \SL_2(\R)$. The Weil representation $\omega_\infty$ in the Schr\"odinger model is a representation of $G'(1)(\R)\times G(\R)$  on $S(V(\R))$ (see e. g. \cite{We1}, \cite{Ku3} or \cite{St1}). The group $G(\R)$ acts on $S(V(\R))$ in a natural way by
\begin{equation}\label{eq:weil_rep_orth_group}
  \omega(g)\varphi(x) = \varphi(g^{-1}x).
\end{equation}
It suffices to describe the action of $G'(1)(\R)$ on its generators:
\begin{equation}\label{eq:weil_rep_sl_2}
  \begin{split}
    &    \omega_\infty(m(a))\varphi(x) = a^{m/2}\varphi(ax) \text{ for } a> 0 \text{ and } m(a) = \kzxz{a}{0}{0}{a^{-1}},  \\
    & \omega_\infty(n(b))\varphi(x) = e^{\pi ib(x,x)}\varphi(x) \text{ with } n(b) = \kzxz{1}{b}{0}{1} \text{ and } \\
    & \omega_\infty(S)\varphi(x) = e(-\sig(L)/8)\widehat{\varphi}(-x) \text{ with } S= \kzxz{0}{-1}{1}{0},
  \end{split}
\end{equation}
where $\widehat{\varphi}(y) = \int_{V(\R)}\varphi(t)e^{2\pi i (x,y)}dt$ is the Fourier transform of $\varphi$. Note that $\omega_\infty$ can be defined in a much more general adelic setting as a representation of $G'(n)(\A)\times G(\A)$ on the space of Schwartz-Bruhat functions $S(V(\A)^n)$, where $\A$ is the ring of adeles of $\Q$ and $G'(n)(\A) = \Sp(n,\A)$ (see the literature cited above). We denote this more general Weil representation with $\omega$ and use the notation $K'(n) = K'(n)_\infty\prod_pK'(n)_p$ with $K'(n)_p= G'(n)(\Z_p)$ for the maximal compact subgroup in $G'(n)(\A)$. Here,
\begin{equation}\label{eq:max_comp_SP}
  K'(n)_\infty = \left\{k=\zxz{a}{b}{-b}{a}\in \Sp(n,\R) \; \bigg\vert\; {\bf k} = a + ib\in U(n)\right\},
  \end{equation}
where $U(n)$ means the unitary group. 

Theta series  will play vital role in this paper. Associated to $\varphi\in S(V(\R))$ and $\mu\in A$ we define
\begin{equation}\label{eq:theta_schwarz_func}
\vartheta(g',\varphi,\mu) = \sum_{\lambda\in \mu + L}\omega_\infty(g')\varphi(\lambda).
\end{equation}
As $\omega_\infty(g')\varphi$ is rapidly decreasing, $\vartheta(g',\varphi,\mu)$ is well defined. 
The standard majorant $(\cdot,\cdot)_z$  associated to $z\in D$ is defined by
\begin{equation}\label{eq:standard_majorant}
  (x,x)_z = (x_{z^\perp}, x_{z^\perp}) - (x_z,x_z),
\end{equation}
where $x = x_{z^\perp}+ x_z$ is the decomposition of $x$ with respect to $V=z^\perp + z$.
Now let the standard Gaussian on $V(\R)\times D$ be
  \begin{equation}\label{eq:standard_gaussion_1_variable}
  \varphi_\infty^{p,q}(x,z) = \exp\left(-\pi(x,x)_z\right),
  \end{equation}
  $(p,q)$ emphasizing the type of $Q$. Since  $(x,x)_z$ is positive for all $x\in V$, the Gaussian is rapidly decreasing and thus an element of $S(V(\R))$.
  More generally, for ${\bf x}=(x_1,\dots,x_n)\in V(\R)^n$ 
  \begin{equation}\label{eq:standard_gaussion_n_variables}
  \varphi_\infty^{p,q}({\bf x},z) = \exp\left(-\pi\sum_{i=1}^n(x_i,x_i)_z\right)
  \end{equation}
  is the standard Gaussian on $V(\R)^n\times D$. 
  
  We choose for $g'\in G'(1)(\R)$ the matrix $g_\tau = n(u)m(\sqrt{v})$. It  moves the base point $i$ to the element $\tau=u+iv$ in the upper half plane $\H$. With this choice of $g'$ and $\varphi$ we can define a Siegel theta function by means of $\vartheta$:
  \begin{equation}\label{eq:theta_gaussian}
    \begin{split}
      \theta(\tau,z,\mu) &= \vartheta(g_\tau,\mu,\varphi_\infty^{p,q}(\cdot,z))\\
      & = \sum_{\lambda\in L+\mu}\varphi_\infty^{p,q}(\lambda,\tau,z).
      \end{split}
  \end{equation}
  with
  \begin{equation}\label{eq:gaussian}
    \begin{split}
    \varphi_\infty^{p,q}(\lambda,\tau,z) &= v^{q/2}\exp\left(\pi i((\lambda,\lambda)u+(\lambda,\lambda)_ziv)\right)\\
    &=v^{q/2}e\left(Q(\lambda_{z^\perp})\tau + Q(\lambda_z)\overline{\tau}\right)
    \end{split}
  \end{equation}
  The last equation of \eqref{eq:theta_gaussian} can be obtained immediately by employing the explicit formulas of the Weil representation (cf.\eqref{eq:weil_rep_sl_2}). 
  It can be shown (see e. g. \cite{B}, Theorem 4.1) that the vector valued theta series
  \begin{equation}
    \Theta_A(\tau,z,\varphi_\infty^{p,p}) = \sum_{\mu\in A}\theta(\tau,z,\mu)
  \end{equation}
  is a real-analytic function with respect to both variables $\tau$ and $z$. Additionally, it transforms with respect to the modular group
\[
\Gamma=\SL_2(\Z)
\]
like a vector valued modular form of weight $(p-q)/2$ for the finite Weil representation (see Section \ref{subsec:weak_maas_forms} for the definition of vector valued modular forms) and is invariant under the action of $SO(V)(\R)$.  
Later in the paper, we will utilize the notation $\mathbbm{1}_M$ for the characteristic function of a set $M$ and we denote by $\iota$ the embedding
\begin{equation}\label{eq:embed_sl}
   \iota: G'(1)\times G'(1) \rightarrow G'(2), \quad \iota\left(\begin{pmatrix}a&b\\c&d\end{pmatrix},\begin{pmatrix}a'&b'\\c'&d'\end{pmatrix}\right)=\begin{pmatrix}a&0&b&0\\ 0&a'&0&b'\\c&0&d&0\\0&c'&0&d'\end{pmatrix}. 
\end{equation}
In some sections of the present work the variable $p$ has two meanings. On the hand, it is part of the type $(p,q)$ of the lattice $L$. On the other hand, $p$ stands for a prime $p$ parametrizing a local Archimedean place. However, there is no danger of confusion since it is always clear from the context what the meaning of $p$ is. 

\section{The finite Weil representation and weak Maass forms}\label{subsec:weak_maas_forms}
In this section we recapitulate the background material to define harmonic weak Maass forms as introduced e. g. \cite{BF1}.  We assume the conventions and  notation from Section \ref{sec:preliminaries}. In \cite{Br1} Bruinier explained how to extended the Borcherds lift to this type of modular forms. They can be seen as a natural generalization of weak holomorphic modular forms transforming according to the finite Weil representation. Later on, weak Maass forms will play a crucial role in the proof of the subjectivity of the Borcherds lift.

Recall from Section \ref{sec:preliminaries} that the discriminant group $A=L'/L$ equipped with the modulo 1 reduction of $Q$ defines a quadratic module. 
Associated to $A$ there is  a representation $\rho_A$ of $\Gamma$ on the group ring $\C[A]$, which we call the ``finite'' Weil representation.  We denote the standard basis of $\C[A]$  by $\{\frake_{\mu}\}_{\mu\in A}$.
The standard scalar product on $\C[A]$ is given by
  \begin{equation}\label{eq:scalar_product_group_ring}
    \left\langle \sum_{\mu\in A}a_\mu\frake_\mu,\sum_{\mu\in A}b_\mu\frake_\mu \right\rangle =\sum_{\mu\in A}a_\mu \overline{b_\mu}.
  \end{equation}
  Note that the group rings $\C[A^2]$ and $\C[A]$ are related by the following isomorphism
  \begin{equation}\label{eq:tensorproduct_group_ring}
    \C[A^2]\longrightarrow \C[A]\otimes\C[A],\quad \frake_{(\mu,\nu)}\mapsto \frake_\mu\otimes\frake_\nu.
  \end{equation}
  This will be used later in  Section \ref{sec:inj_kudla_millson}. 
  
As $\Gamma$ is generated by the matrices
\begin{equation}\label{eq:symplectic_group_gen}
S=\zxz{0}{-1}{1}{0},\quad T= \zxz{1}{1}{0}{1}, 
\end{equation}
it is sufficient to define $\rho_A$ by the action on these  generators.

\begin{definition}\label{def:weil_finite_repr}
  The representation $\rho_{A}$ of $\Gamma$ on $\C[A]$, defined by
  \begin{equation}\label{eq:weil_finite_repr}
\begin{split}
&\rho_A(T)\frake_\mu := e(bQ(\mu))\frake_\mu,\\
& \rho_A(S)\frake_\mu:=\frac{e(-\sig(A)/8)}{|A|^{1/2}}\sum_{\nu\in A}e(-(\nu,\mu))\frake_{\nu},
\end{split}
\end{equation}
  is called Weil representation. Note that $\rho_A$ can be identified with a subrepresentation of the Weil representation $\omega$.  We denote the dual representation of $\rho_A$ by $\rho_A^*$. It is obtained from $\rho_A$ by passing from $(L,(\cdot,\cdot))$ to $(L,-(\cdot,\cdot))$ or simply by complex conjugation when considered as matrices. Thus $\displaystyle \rho_A^* = \rho_{A^-} = \overline{\rho_A}$,
  where $A^-$ means the discriminant group $A$ equipped with $-(\cdot,\cdot)$. 
\end{definition}
Let $Z=\kzxz{-1}{0}{0}{-1}$. The action of $Z$ is given by
\begin{equation}\label{eq:weilz}
\rho_A(Z)(\frake_\mu)=e(-\sig(A)/4)\frake_{-\mu}.
\end{equation}
We denote by $N$ the level of the lattice $L$. It is the smallest positive integer such that $NQ(\lambda)\in \Z$ for all $\lambda\in L'$. For the rest of this paper we suppose that $N$ is {\it odd}.
For later use we introduce a Gauss sum associated to the discriminant group $A$. For an integer $d$ we write
\begin{equation}\label{eq:Gauss_sum}
  g_d(A) = \sum_{\mu\in A}e(dQ(\mu))
  \end{equation}
and put $g(A) = g_1(A)$. By Milgram's formula we have $g(A)= \sqrt{|A|}e(\sig(A)/8)$.

We now define vector valued modular forms of type $\rho_A$. With respect
to the standard basis of $\C[A]$ a function $f: \H\rightarrow \C[A]$ can be written in the form
\[
f(\tau) = \sum_{\mu\in A}f_\mu(\tau)\frake_\mu.
\]
The following operator generalises the usual Petersson slash operator to the space of all those functions. For $\kappa\in \Z$ we put
\begin{equation}\label{eq:slash_operator}
f\mid_{\kappa,A}\gamma =
j(\gamma,\tau)^{-\kappa}\rho_A(\gamma)^{-1}f(\gamma\tau),
\end{equation}
where
\[
j(g,\tau) = \det(g)^{-1/2}(c\tau+d)
\]
is the usual automorphy factor if $g=\kzxz{a}{b}{c}{d}\in \GL_2^+(\R)$.

A holomorphic function $f: \H\rightarrow \C[A]$ is called a {\it weakly holomorphic} modular
form of weight $ \kappa $ and type $ \rho_A $ for $ \Gamma $ if
$ f\mid_{\kappa,A}\gamma= f $ for all $\gamma\in \Gamma$,
and if $f$ is meromorphic  at the cusp $\infty$. Here the last condition means that $f$ admits a Fourier expansion of the form
\begin{equation}\label{eq:fourier_exp_weakly_holom}
\sum_{\mu\in A}\sum_{n\in \Z+Q(\mu)}c(\mu,n)e(n\tau),
\end{equation}
where all but finitely many Fourier coefficients $c(\mu, n)$ with $n <  0$ vanish. If in addition $c(\mu,n) = 0$ for all $n < 0$ ($n\le 0$), we call the corresponding modular form {\it holomorphic} (a {\it cusp form}). We denote by $M^!_{\kappa,A}$, the
space of all weakly holomorphic modular forms, by $M_{\kappa,A}$ the space of holomorphic modular forms and  by $S_{\kappa,A}$ the subspace cusp forms. We write $M^!_{\kappa,A^-}, M_{\kappa,A^-}$ and $S_{\kappa,A^-}$ for the corresponding spaces with respect to the dual Weil representation $\rho_{A^-}$.  For more details see e.g. \cite{Br1}. Note that formula \eqref{eq:weilz}
implies that $M_{k,A} = \{0\}$ unless
\begin{equation}\label{eq:sig_weight}
2\kappa\equiv \sig(L)\pmod{2}.
\end{equation}
Therefore, if the signature of $L$ is even, only non-trivial spaces of integral
weight can occur. 

The Petersson scalar product on $S_{\kappa,A}$ is given by
\begin{equation}\label{eq:petersson_scalar}
(f,g) =
\int_{\Gamma\backslash \H}\langle f(\tau), g(\tau)\rangle \im \tau^\kappa d\mu(\tau)
\end{equation}
where
\[
d\mu(\tau) = \frac{du\;dv}{v^2}
\]
denotes the hyperbolic volume element with  $\displaystyle \tau = u + iv$.
In view of \eqref{eq:tensorproduct_group_ring} we have for $f\in S_{\kappa_1,A}$ and $g\in S_{\kappa_2,A}$  that
\begin{equation}\label{eq:tensor_product_forms}
  f\otimes g = \sum_{\mu,\nu\in A}f_\mu g_\nu \frake_\mu\otimes\frake_\nu \in S_{\kappa_1+\kappa_2,A^2}. 
\end{equation}
  
Following \cite{BF1}, a  harmonic weak Maass forms of weight $\kappa$ with representation $\rho_A$ is a twice continuously differentiable function $f: \H\longrightarrow \C$ such that
\begin{enumerate}
\item[i)]
  $f\mid_{\kappa,A} \gamma = f$ for all $\gamma \in \Gamma$, 
\item[ii)]
  $\Delta_\kappa f = 0$, where
  \begin{equation}\label{eq:laplace_op_wt_k}
    \Delta_\kappa = - v^2\left(\frac{\partial^2}{\partial u^2} + \frac{\partial^2}{\partial v^2}\right) + iky\left(\frac{\partial}{\partial u} + i\frac{\partial}{\partial v}\right) 
  \end{equation}
  is the hyperbolic  Laplace operator of weight $\kappa$.
\item[iii)]
  $f(\tau) = O(e^{\epsilon v})$ for $v \rightarrow \infty$ for some $\epsilon > 0$ (uniformly in $u$). 
\end{enumerate}
By $H_{\kappa,A}$ we mean the space of weak Maass forms. According to \cite{BF1} each such Maass form possesses a unique decomposition $f= f^+ + f^-$ into a homomorphic part $f^+$ and a non-holomorphic part $f^-$ having Fourier expansions of the form
\begin{align*}
 & f^+(\tau) = \sum_{\mu\in A}\sum_{\substack{n\in \Z+Q(\mu)\\n\gg -\infty}}c^+(\mu,n)e(n\tau) \text{ and }\\
 & f^-(\tau) =  \sum_{\mu\in A}\left(c^-(\mu,0) + \sum_{\substack{n\in \Z+Q(\mu)\bs\{0\}\\n\ll \infty}}c^-(\mu,n)H_\kappa(2\pi n v)e(n u) \right).
  \end{align*}
Here $\displaystyle H_\kappa(w)$ is given by $e^{-w}\int_{-2w}^\infty e^{-t}t^{-\kappa}dt$. For  $w<0$ we have $H_\kappa(w) = e^{-w}\Gamma(1-\kappa,2w)$, where $\Gamma(a,x)$ means the incomplete Gamma function. If $w>0$, the integral converges only for $\kappa<1$. But it can be analytically continued to all $\kappa\in \C$ the same way as the Gamma function.

For $\tau=u+iv\in\H$ we define the Maass lowering and raising operators of weight $\kappa$  on non-holomorphic modular forms by
\begin{equation}\label{eq:maass_lowering_raising}
L_\kappa = -2iv^2\frac{\partial}{\partial \overline{\tau}},\quad R_\kappa = 2i\frac{\partial}{\partial \tau} + \kappa v^{-1}.
\end{equation}
It is shown in \cite{BF1}, Prop. 3.2, that the differential operator $\displaystyle \xi_\kappa(f)(\tau) = v^{\kappa-2}\overline{L_\kappa(f)(\tau)}$ maps weak Maass forms of weight $\kappa$ to weakly holomorphic modular forms of weight $2-\kappa$ transforming with $\rho_{A^-}$ . Denote with $H^+_{\kappa,A}$ the subspace of weak Maass forms which are mapped by $\xi_\kappa$ to the space of cusp forms $S_{2-\kappa,A^-}$. It is proved in \cite{BF1}, Theorem 3.7, that $\xi_\kappa$ is surjective, which implies that the sequence 
\begin{equation}\label{eq:exact_sequence}
  0 \longrightarrow M^!_{\kappa,A}\longrightarrow H^+_{\kappa,A}\overset{\xi_\kappa}{\longrightarrow}  S_{2-\kappa,A^-}\longrightarrow 0.
\end{equation}
is exact.
By means of \cite{BF1}, Lemma 3.1, we see that  the Fourier expansion of $f\in H_{\kappa,A}^+$ is given by
\begin{equation}\label{eq:fourier_exp_xi_k}
  f(\tau) = \sum_{\mu\in A}\left(\sum_{\substack{n\in \Z+Q(\mu)\\n\gg -\infty}}c^+(\mu,n)e(n\tau) + \sum_{\substack{n\in \Z+Q(\mu)\bs\{0\}\\n< 0}}c^-(\mu,n)\Gamma(1-\kappa,4\pi |n| v)e(n \tau)\right).
  \end{equation}
The principal part of such a Maass form $f$ is then given by
\begin{equation}\label{eq:principal_part_weak_maass}
  P_f(\tau) = \sum_{\mu\in A}\sum_{\substack{n\in \Z+Q(\mu)\\n<0}}c^+(\mu,n)e(n\tau).
\end{equation}

\section{Special Schwartz forms}\label{subsec:schwarz_forms} 
 In this paper the Kudla-Millson form $\varphi_{\KM}$ (\cite{KM1}) and the more general Schwartz form $\varphi_{q,l}$ (\cite{FM}) play a prominent role as they constitute the theta  kernel of the Kudla-Millson lift $\Lambda$.  The proof of the injectivity of $\Lambda$ makes use of the related Schwarz functions $\phi_{q,l}$ and $\xi$ and some of their properties.

Schwartz forms are a generalization of Schwartz functions. In this section we review the construction of the Schwartz forms $\varphi_{\KM}$, $\varphi_{q,l}$ and the aforementioned Schwartz functions $\phi_{q,l}$ and $\xi$. We additionally list some of their fundamental properties. Our main sources are \cite{KM1}, \cite{FM} and \cite{BF}.

Schwartz forms are Schwartz functions on $V$ with values in the differential $q$-forms on the Grassmannian $D$. They can be considered as elements of $S(V)\otimes \calA^{q}(D)$. Note that  $G(\R)$ acts  on elements of $S(V)\otimes \calA^q(D)$ via
\begin{equation}\label{eq:action_G}
L_g^*\varphi(gx,z),
\end{equation}
where $L_g^*$ means the pullback on  $\calA^q(D)$ induced by left translations of $G(\R)$ on $D$. 

The construction of the Kudla-Millson Schwartz form $\varphi_{\KM}$ makes use of the isomorphism
\begin{equation}\label{eq:eval_basepoint}
\left[S(V)\otimes \calA^r(D)\right]^{G(\R)} \cong \left[S(V)\otimes \bigwedge^r\frakp^*\right]^{K_\infty},
\end{equation}
which is given by mapping any $\varphi(x,z)\in \left[S(V)\otimes \calA^r(D)\right]^{G(\R)}$ to $\varphi(x, z_0)\in \left[S(V)\otimes \bigwedge^r\frakp^*\right]^{K_\infty}$, where $z_0\in D$ denotes a fixed base point 
 and  $\frakp$ is part of the Cartan decomposition  $\frakg = \frakp + \frakt$  of the Lie algebra $\frakg$ of $G(\R)$ with $\frakt$ being the Lie algebra of $K_\infty$. The isomorphism \eqref{eq:eval_basepoint} is based on the well known fact (see e. g. \cite{He}, Chap. IV, Sec. 3 ) that $\frakg/\frakt\cong\frakp$ is isomorphic to the tangent space $T(D)_{z_0}$ at the chosen base point $z_0$.  As usual, $\frakp^*$ means the dual space of $\frakp$ and $\bigwedge^r$ the $r$-fold exterior product of $\frakp^*$. 
In view of \ref{eq:eval_basepoint}, it suffices to specify $\varphi_{\KM}$ on $\frakp^*$. To this end, we utilize the explicit description of $\frakp$ by 
\[
\left\{\zxz{0}{X}{X^t}{0}\;\bigg\vert \; X\in M_{p,q}(\R)\right\}\cong M_{p,q}(\R). 
  \]
  By means of this isomorphism, the standard basis of $M_{p,q}(\R)$ gives rise to  a basis $X_{\alpha,\mu}$ of $\frakp$ with respect to the above chosen basis of $V$.  For $1\le \alpha\le p$ and $p+1\le \mu\le p + q$ we have
  \[
  X_{\alpha,\mu}(v_i) = \begin{cases}v_\mu,& i = \alpha,\\ v_\alpha, & i=\mu,\\ 0,& \text{otherwise}.\end{cases}
  \]
 Moreover, by $\{\omega_{\alpha,\mu}\;|\;\alpha=1,\dots, p \text{ and } \mu = p+1,\dots, p+q \}\subset  \frakp^*$ we mean the associated dual basis. 

Let $\varphi_{\infty}^{p,q}$ be the standard Gaussian (see Section \ref{sec:preliminaries}). For the rest of this section we drop the superscript $p,q$ to lighten the notation. It can be verified that
\[
\varphi_\infty(gx,gz) = \varphi_\infty(x, z)
\]
for all $g\in G(\R)$, i. e.
\[
\varphi_\infty\in \left[S(V)\otimes C^\infty(D)\right]^{G(\R)}.
\]
For the sake of clarity we write
\begin{equation}\label{eq:gaussian_base_point}
  \varphi_0(x) \text{ instead of } \varphi_\infty(x,z_0).
\end{equation}
  Then $\varphi_{\KM}$ is defined by (see \cite{KM1}, Chap. 3 and Chap. 5 and \cite{FM}, Chap. 5.2) by applying the Howe operator
\begin{equation}\label{eq:kudla_millson_diff_op}
\begin{split}
  &  \calD: S(V)\otimes {\bigwedge}^{\bullet}(\frakp^*)\longrightarrow  S(V)\otimes {\bigwedge}^{\bullet + q}(\frakp^*),\\
&  \calD = \frac{1}{2^{q/2}}\prod_{\mu=p+1}^m\left[\sum_{\alpha = 1}^p\left(x_\alpha - \frac{1}{2\pi}\frac{\partial}{\partial x_\alpha}\right)\otimes A(\omega_{\alpha\mu})\right]
\end{split}
\end{equation}
to $\varphi_0\otimes 1\in \left[S(V)\otimes {\bigwedge}^0(\frakp^*)\right]^K$:
\begin{equation}\label{eq:kudla_millson_form}
\varphi_{KM}  = \calD(\varphi_0\otimes 1).
\end{equation}
Here    $A(\omega_{\alpha\mu})$ denotes the exterior left multiplication by $\omega_{\alpha \mu}$. 


The definition of the more general Schwartz form $\varphi_{q,l}$ can be found in \cite{FM}, Chap. 5.2. (for $n=1$ in our case). As a differential form on $D$,  $\varphi_{q,l}$  takes values in $S(V)\otimes \Sym^l(V)$, where $\Sym^l(V)$ is the $l$-th symmetric power of $V$. As pointed out in \cite{BMM}, Lemma 8.2 and Theorem 8.3, the isomorphism in \eqref{eq:eval_basepoint} extends to these more general forms. We have
\begin{equation}\label{eq:eval_basepoint_vec_val}
  \left[S(V)\otimes \calA^q(D)\otimes \Sym^l(V)\right]^{G(\R)}\cong \left[S(V)\otimes \bigwedge^q \frakp^\ast \otimes \Sym^l(V)\right]^{K_\infty}.
  \end{equation}
Consider the operator 
\begin{equation}\label{eq:funke_millson_diff_op}
\calD' = \left[\frac{1}{2}\sum_{\alpha=1}^p\left(x_\alpha - \frac{1}{2\pi}\frac{\partial}{\partial x_\alpha}\right)\otimes 1\otimes A(v_\alpha)\right]^l \in \End_\C\left(S(V)\otimes {\bigwedge}^\bullet(\frakp^*)\otimes \Sym^l(V)\right),
\end{equation}
where $A(v_\alpha)$ means the multiplication in $\Sym^l(V)$ by $v_\alpha$. 
Then $\varphi_{q,l}$
is obtained from $\varphi_{\KM} \in \left[S(V)\otimes \bigwedge^q \frakp^\ast\right]^{K_\infty}$ by
\begin{equation}\label{eq:funke_millson_form}
  \varphi_{q,l} = \calD'(\varphi_{KM}).
\end{equation}
If $l$ is equal to zero, $\varphi_{q,l}$ simplifies to $\varphi_{\KM}$.
Setting $\calD_\alpha = \left(x_\alpha - \frac{1}{2\pi}\frac{\partial}{\partial x_\alpha}\right)$, we may write $\varphi_{q,l}$ in a more explicit form
\begin{equation}\label{eq:varphi_q_l}
\varphi_{q,l}(x)=  \sum_{\substack{\alpha_1,\dots,\alpha_q=1\\\beta_1,\dots,\beta_l=1}}^p \prod_{i=1}^q\prod_{j=1}^l \calD_{\alpha_i}\calD_{\beta_j}\varphi_0(x)\otimes\left( \omega_{\alpha_1 p+1}\wedge \dots \wedge \omega_{\alpha_q p+q}\right) \otimes \left(v_{\beta_1}\otimes \dots \otimes v_{\beta_l}\right).
\end{equation}


The following theorem summarizes the most fundamental properties of $\varphi_{q,l}$ (and thereby of $\varphi_{\KM}$). They are proved in \cite{FM}, Section 6 and \cite{KM1}, Thm. 3.1, see also \cite{BF} Section 3.
\begin{theorem}\label{thm:properties_varphi_q_l}
The Schwartz form $\varphi_{q,l}$ satisfies: 
\begin{enumerate}
\item[i)]
$\varphi_{q,l}$ is invariant under the action of $K_\infty$, i. e. $\varphi_{q,l}\in \left[S(V)\otimes {\bigwedge}^q(\frakp^*)\otimes \Sym^l(V)\right]^{K_\infty}$.
\item[ii)]
  $\varphi_{q,l}$ is an eigenvector of weight $\frac{m}{2}+l$ under the action of $K_\infty'(1)$, i. e.
  \[
  \omega(k)\varphi_{q,l} = \det({\bf k})^{\frac{m}{2}+l}\varphi_{q,l},
  \]
  ${\bf k} \in U(1)$ corresponding to $k\in K_\infty'(1)$ (cf. \eqref{eq:max_comp_SP}).
\item[iii)]
  $\varphi_{q,l}$ is a closed differential $q$- form on $D$ with values in $\Sym^l(V)$. 
\end{enumerate}
\end{theorem}

The Schwartz function $\phi_{q,l}\in S(V(\R)^2)$  is constructed by applying the Hodge star operator $\ast$  on $D$ to $\varphi_{q,l}$. More precisely,

\begin{equation}\label{eq:phi_q_l}
\phi_{q,l}((x_1,x_2),z)\mu =  \varphi_{q,l}(x_1,z)\wedge \ast(\varphi_{q,l})(x_2,z),
\end{equation}
where $\mu$ is the $G(\R)$-invariant volume form on $D$ induced by the Riemann metric coming from the Killing form on $\frakg$.

The definition of the Schwartz function $\xi$ is more involved:
Let $\frakg'$ be the complexified Lie algebra of $G'(2)(\R)$. The Harish-Chandra decomposition of $\frakg'$ is then given by
\[
\frakg' = \frakp_+ \oplus \frakp_- \oplus \frakt'
\]
with $\frakt'$ being the complexified Lie algebra of $K'(2)_\infty$. We have an explicit description of $\frakp_+$:
\[
\frakp_+ = \left\{p_+(X) = \frac{1}{2}\zxz{X}{iX}{iX}{-X}\;\bigg\vert\; X\in M_{2,2}(\C)\; X^t = X\right\}.
\]
It is generated by
\[
R_{11} = p_+\kzxz{1}{0}{0}{0},\quad R_{22} = p_+\kzxz{0}{0}{0}{1},\quad R_{12} = \frac{1}{2}p_+\kzxz{0}{1}{1}{0}.
\]
It is well known (see e. g. \cite{Bu}, Chap. 2.1 and Chap. 2.2) that $R_1 = \iota(R,0)$ and $R_2= \iota(0,R)$, where
\[
R = \frac{1}{2}\zxz{1}{i}{i}{-1}
\]
corresponds to the Maass raising operator.
Now we are in the position to specify the Schwartz function $\xi\in S(V(\R)^2)$. We set
\begin{equation}\label{eq:xi_q_l}
      \xi = \frac{p^l(-1)^{q+l}}{2^l\pi^{q+l}}\omega(\alpha_{q+l})\varphi_0
    \end{equation}
    with
    \[
    \alpha_{q+l} =
    \begin{cases}
      \left(R_{12}^2 - R_{11}^2R_{22}^2\right)^{\frac{q+l}{2}}, & \text{ if } q+l \text{ is even}, \\
      R_{12}\left(R_{12}^2 - R_{11}^2R_{22}^2\right)^{[\frac{q+l}{2}]}, & \text{ if } q+l \text{ is odd}
      \end{cases}
    \]
    being an element of  $\frakp_+$. 

    Note that the functions $\phi_{q,l}$ and $\xi$ are related by the identity
     \begin{equation}\label{eq:rel_xi_phi_q_l}
    \phi_{q,l} = \xi + \omega(R_{11})\omega(R_{22})\psi,
     \end{equation}
     where $\psi$ is specified in \cite{BF}, Prop. 3.9.
   
    
We close this section with
\begin{lemma}[\cite{BF}, Section 3]\label{lem:properties_xi}
We have
  \begin{enumerate}
  \item[i)]
    $\xi$ is $K_\infty$-invariant, i. e. an element in $\left[S(V(\R)^2)\otimes {\bigwedge}^0(\frakp^*)\right]^{K_\infty}$,
  \item[ii)]
    $\xi$ is a weight $\frac{m}{2}+l$ eigenvector with respect to the action of $K'_\infty(2)$ via the Weil representation.
  \item[iii)]
    $\xi$ vanishes identically if and only if $p=1$ and $q+l>1$. 
\end{enumerate}
\end{lemma}



\section{The Siegel-Weil formula}\label{subsec:siegel_weil}
This section summarizes the necessary background to state the Siegel-Weil formula in a vector valued setup. We mainly follow \cite{Ku2} and \cite{Ku3}. This formula connects the integral over a theta function associated to a Schwartz-Bruhat function  and a special value of an adelic Eisenstein series. It is stated in a global and more general setting compared to Section \ref{subsec:schwarz_forms}.

Recall that $S(V(\A)^n)$ is the space of Schwartz-Bruhat functions, $G(\A)=O(V(\A))$ and $G'(n)(\A) = \Sp(n,\A)$.
Within $G'(n)(\A)$ we have the following subgroups
\begin{align}\label{def:subgroups_ga}
 & M(\A)=\left\{m(a)=\zxz{a}{0}{0}{(a^{-1})^t}\; \big| \; a\in \GL_n(\A)\right\},\\
 & N(\A)=\left\{n(b)=\zxz{1}{b}{0}{1}\; \big|\; b\in \Sym_n(\A) \right\}.  
\end{align}
These define the Siegel parabolic subgroup $P(\A)=N(\A)M(\A)$ of $G'(n)(\A)$,  which is part of the Iwasawa decomposition
\begin{equation}\label{eq:iwasawa}
G'(n)(\A)=N(\A)M(\A)K'(n)
\end{equation}
of $G'(n)(\A)$, 
where $K'(n)=\prod_v K'_v(n)$ is the maximal compact subgroup of $G'(n)(\A)$. 
Here for a non-Archimedean place $v=p$ the group $K'_p(n)$ is given by $\Sp(n,\Z_p)$ and   $K'(n)_\infty$ is the maximal compact in $G'(n)(\R)$ (cf. \eqref{eq:max_comp_SP}).

For $\varphi\in S(V(\A)^n)$ we have a generalization of the theta function \eqref{eq:theta_schwarz_func} to the genus $n$ in the adelic setting:
\begin{equation}\label{eq:theta_function_adelic}
  \vartheta(g, h, \varphi) = \sum_{{\bf x}\in V(\Q)^n}\omega(g)\varphi(h^{-1}{\bf x}),
\end{equation}
where $g\in G'(n)(\A)$ and $h\in G(\A)$. Here $G'(n)(\A)$ acts on $S(V(\A)^n)$ via  the global Weil representation $\omega$ (see e. g. \cite{We1}, \cite{Ku3} or \cite{St1}).
This action commutes with the action
\[
\varphi({\bf x})\mapsto \varphi(h^{-1}{\bf x})
\]
of $G(\A)$. Sometimes the action of $G(\A)$ is written as $\omega(h)\varphi$.  It can be shown that $\vartheta(g,h,\varphi)$ is left invariant under $G'(n)(\Q)\times G(\Q)$, slowly increasing on $(G'(n)(\Q)\times G(\Q))\bs(G'(n)(\A)\times G(\A))$. Moreover, it defines a smooth function on $G'(n)(\A)\times G(\A)$.
The integral
\begin{equation}\label{eq:theta_integral_adelic}
  I(g,\varphi)= \int_{G(\Q)\bs G(\A)}\vartheta(g,h,\varphi)dh
\end{equation}
can be interpreted as the average value of $\vartheta$ with respect to $G(\A)$. Here $dh$ is the Haar measure on $G(\Q)\bs G(\A)$ normalized such that  $\vol(G(\Q)\bs G(\A)) = 1$, thus half of the Tamagawa measure. By Weil's convergence criterion (see \cite{We2}, Chap. VI, Prop. 8), the integral in \eqref{eq:theta_integral_adelic} is absolutely convergent for all $\varphi$ whenever either $V$ is anisotropic or the rank of $L$ and the Witt index of $V$ satisfy the inequality $m-r_0 > n+1$. 
Also, if $\varphi$ is $K'(n)$-finite, $g\mapsto I(g,\varphi)$ defines an automorphic form on $G'(n)(\A)$ (\cite{Ku3}).

To describe the Eisenstein series involved in the Siegel-Weil formula, we let
\begin{equation}
\chi_V: \A^\times / \Q^\times \longrightarrow\C; \quad x\mapsto \chi_V(x)=(x,(-1)^{m/2}\det(V))_{\A}
\end{equation}
the quadratic character defined by the global Hilbert symbol $(\cdot,\cdot)_\A$, where $\det(V)$ is the Gram determinant of $V$.

For $s\in \C$, we denote by $I_n(s,\chi_V)$ the normalized induced representation from $P(\A)$ to $G'(n)(\A)$. It consists of smooth function $g\mapsto \Phi(g, s)$ on $G'(n)(\A)$ satisfying
\begin{equation}\label{eq:induced_representation}
  \Phi(n(b)m(a)g,s) = \chi_V(\det(a))|\det(a)|^{s+\rho_n}\Phi(g,s)
\end{equation}
for all $a\in \GL_n(\A)$ and $b\in \Sym_n(\A)$, where 
\begin{equation}
\rho_n=\frac{n+1}{2}. 
\end{equation}
An element $\Phi\in I_n(s,\chi_V)$ is called a standard section if its restriction to $K'(n)$ is independent of $s$.
For any $\Phi\in I_n(s,\chi_V)$ and $g\in G'(n)(\A)$, we define the (adelic) Siegel Eisenstein series of genus $n$ by
\begin{equation}\label{eq:scalar_valued_eisenstein_series}
  E(g,s,\Phi) = \sum_{\gamma\in P(\Q)\bs G'(n)(\Q)}\Phi(\gamma g,s).
\end{equation}
If $\Phi$ is a standard section, one can proof that $E(g,s,\Phi)$  converges absolutely for $\re(s) > \rho_n$ and thereby defines and automorphic form on $G'(n)(\A)$ provided $\Phi$ is $K'(n)$-finite. Moreover, $E(g,s,\Phi)$ can be continued meromorphically to the whole $s$-plane (for these results see e. g. \cite{A})) and satisfies a functional equation.

One way to construct a standard section $\Phi$ is by means of the intertwining map
\begin{equation}
  \lambda: S(V(\A)^n)\longrightarrow I_n(s_0, \chi_V), \quad \varphi \mapsto \lambda(\varphi)(g,s_0) = \Phi(g,s_0) = (\omega(g)\varphi)(0),
\end{equation}
where
\begin{equation}\label{eq:s_0}
  s_0 = \frac{m}{2}-\rho_n
\end{equation}
and $\omega$ is the (adelic) Weil representation.
Using the Iwasawa decomposition, write $g\in G(\A)$ as $g= n(b)m(a)k$ and put
\[
|a(g)| = |\det(a)|_\A.
\]
It can then be proved that $\lambda(\varphi)\in I_n(s_0,\chi_V)$ has a unique extension to a standard section of $I_n(s,\chi_V)$ given by
\begin{equation}\label{eq:extension_standard_section}
  \Phi(g,s) = \lambda(\varphi)(g,s) = |a(g)|^{s-s_0}(\omega(g)\varphi)(0)\in I_n(s,\chi_V). 
\end{equation}

The Siegel-Weil formula was originally stated by Weil in a very general manner for dual reductive pairs (see \cite{We2}, Chap. IV, Theor\`eme 5). Here, we consider the dual pair $(\Sp(n), O(V))$ and follow \cite{BF}, \cite{KR1} and \cite{KR2}.
\begin{theorem}\label{thm:siegel_weil}
 Assume that $r_0=0$ or $m-r_0 > n+1$. Then for all $K'(n)$-finite $\varphi\in S(V(\A)^n)$
  \begin{enumerate}
  \item[i)]
    $E(g,s,\varphi)$ is holomorphic at $s = s_0$ and
  \item[ii)]
    $E(g,s_0,\lambda(\varphi)) = \alpha I(g,\varphi)$ for all $g\in G(\Q)\bs G(\A)$, where $\alpha = \begin{cases}1, & m > n+1,\\ 2, & m \le n+1.\end{cases}$
    \end{enumerate}
\end{theorem}

We would like to give a vector valued version of Theorem \ref{thm:siegel_weil} in the same vein as in Prop. 2.2 of \cite{BY}, but for lattices of signature $(p,q)$ and the group $\Sp(n)$ (see also \cite{St1}, Section 3, for the following notations and definitions). 
To this end, we define a Schwartz-Bruhat function $\varphi_\mu\in S(V(\A_f)^n)$ associated to $\mu\in (L'/L)^n = A^n$ by
\begin{align}\label{eq:familiy}
  \varphi_\mu = \mathbbm{1}_{\mu + \hat{L}^n} = \prod_{p<\infty} \varphi^{(\mu)}_p = \prod_{p<\infty}\mathbbm{1}_{\mu+L_p^n}. 
\end{align}
Here $L_p=L\otimes \Z_p$, which is the $p$-part of $\hat{L}=L\otimes \hat{\Z}$ with $\hat{\Z}=\prod_{p<\infty}\Z_p$. 
For the Archimedean place we choose the Gaussian $\varphi_0^{p,q}$ evaluated at the base point $z_0\in D$ (see \eqref{eq:gaussian_base_point}).
We associate to $\varphi_0^{p,q}\prod_{p<\infty}\varphi^{(\mu)}_p$ the standard section $\Phi^{(p-q)/2}_\infty(s)$ $\prod_{p<\infty}\Phi_p^{(\mu)}(s) \in I_n(s,\chi_V)$, where  
\begin{equation}\label{eq:norm_section}
  \begin{split}
    & \Phi^{(p-q)/2}_\infty(g_\infty,s)= \lambda_\infty(\varphi^{p,q}_0)(g_\infty,s) \text{ and }\\
    & \Phi_p^{(\mu)}(g_p,s) = \lambda_p(\varphi_p^{(\mu)})(g_p,s)
  \end{split}
\end{equation}
for $g\in G(\A)$. We use the notation
\[
\Phi_\mu(s) = \prod_{p<\infty}\Phi_p^{(\mu)}(s). 
\]

Note that the Schwartz-Bruhat functions $\varphi_\mu,\; \mu\in A^n,$ generate a finite dimensional subspace
\begin{equation}\label{eq:sub_schwartz_bruhat}
  S_L=\bigoplus_{\mu\in A^n}\C\varphi_\mu
\end{equation}
  of $S(V(\A_f)^n)$, which is stable under the action of $G'(n)(\widehat{\Z})$ via the Weil representation (see \cite{St1}, Lemma 3.4 and the discussion after the Lemma).

We are now ready to define a vector valued theta function and a vector valued Eisenstein series. Let $\phi\in S(V(\R)^n)$ with associated standard section $\Phi(g_\infty,s) = \lambda(\phi)(g_\infty,s)$ at the place infinity and $\varphi_\mu$ and $\Phi_\mu$ as above. Then  we put 
\begin{equation}\label{eq:theta_vec_val_1}
  \vartheta_L(g,h,\phi) = \sum_{\mu\in A^n}\vartheta(g,h,\phi\otimes\varphi_\mu)\varphi_\mu 
\end{equation}
and
\begin{equation}\label{eq:eisenstein_adelic_vec_val}
  E_L(g,s,\Phi) = \sum_{\mu\in A^n}E(g,s,\Phi\otimes\Phi_\mu)\varphi_\mu.
  \end{equation}
In terms of $\theta_L$ and $E_L$, based on the same assumptions and notations as in Theorem \ref{thm:siegel_weil}, the Siegel-Weil formula then reads  as follows.

\begin{corollary}\label{prop:siegel_weil_vec_val}
 For $\phi\in S(V(\A)^n)$ with induced section $\Phi = \lambda(\phi)$ we have
  \begin{equation}\label{eq:siegel_weil_vec_val}
    \alpha\int_{O(V)(\Q)\bs O(V)(\A)}\vartheta_L(g,h,\phi)dh = E_L(g,s_0,\Phi).
  \end{equation}
\end{corollary}

We close this section with a result taken from \cite{BF}, which will be important  later in the paper. 
To that end, let $\kappa \in \N$ and  denote with
\begin{equation}\label{eq:standard_section}
  \Phi_\infty^\kappa(s) = \det({\bf k})^{\kappa}
\end{equation}
the standard section of weight $\kappa$ at the place infinity.
It turns out that  the induced section $\Xi$ associated to the Schwartz function $\xi$ (see \eqref{eq:xi_q_l}) is essentially $\Phi_\infty^{m/2+l}$:

\begin{proposition}[\cite{BF},Prop. 3.12]\label{prop:standard_section_xi}
  Let $q, l$ be as in Section \ref{subsec:schwarz_forms} and $q+l$ even. Then
  $\Xi(g,s) = \lambda(\xi)(g,s)$ is the standard section \eqref{eq:standard_section} of weight $m/2 + l$ (at the infinite place), i. e.
  \[
  \Xi(s) = C(s)\Phi_\infty^{m/2+l}(s).
  \]
  Here $C(s)$ is a polynomial, which is non-zero for $s_0=\frac{m}{2}-\frac{3}{2}$ and which vanishes identically for $p=1$.   
  \end{proposition}

\section{A standard $L$-fundtion and zeta function}\label{sec:standard_l_func}
In \cite{St2} and \cite{St3} a zeta function and a standard $L$-function, respectively,  is assigned to a vector valued common Hecke eigenform $f$ of weight $\kappa$ and type $\rho_A$. The usual basic properties for both functions are shown. In this section we briefly review some of the before mentioned material and show that the standard $L$-function is non-zero at some $s\in \C$. This result will turn out to be an important step in the proof of the injectivity of the Kudla-Millson theta lift.  All details can be found in \cite{St2} and \cite{St3}.

Because there is a lot of notation used in this section, we list most of it at the beginning of this section and explain the rest of it within the text or refer to the main sources.
There are several groups involved. The group $\calQ_p$ is a subgroup of $\GL_2(\Q_p)$ and $\calK_p$  means a subgroup of $\GL_2(\Z_p)$. We write $\calM_p$ and $\calD_p$ for a subgroup of diagonal matrices in $\calQ_p$ and $\calK_p$, respectively,  and denote with $N(\Z_p)$ the group $\{\kzxz{1}{x}{0}{1}\; |\; x\in \Z_p\}$. 
Recall that $\omega = \bigotimes_{p\le \infty}\omega_p$ is the global Weil representation of $G'(1)(\A)\times G(\A)$. The Schr\"odinger model of $\omega$ acts on $S(V(\A))$ and leaves the subspace $S_L=\bigotimes_p S_{L_p}$ (see \eqref{eq:sub_schwartz_bruhat}) invariant. $\varphi_p^{(0)}\in S_{L_p}$ is defined by \eqref{eq:familiy} (with $\mu=0$).
As usual, we abbreviate the diagonal matrix $\kzxz{d_1}{0}{0}{d_2}$ with the symbol $m(d_1,d_2)$.

The zeta function $Z(s,f)$ is defined via the eigenvalues $\lambda_f(m(d^2,1))$ of the Hecke operators $T(m(d^2,1))$:
\begin{equation}
  Z(s,f) = \sum_{d\in \N}\lambda_f(m(d^2,1))d^{-2s}. 
\end{equation}
Its analytic properties are linked to those of  a vector valued Siegel Eisenstein series $E_{\kappa,0}^2$ of genus two transforming according to the Weil representation of $G'(2)(\Z)$ by a Rankin-Selberg type integral formula (see \cite{St2}, Thm. 6.4)
\begin{equation}
  \begin{split}
  & \sum_{\lambda\in L'/L}\left(\int_{\Gamma\bs \H}\langle f(\tau)\otimes \frake_\lambda,E_{\kappa,0}^{2}(\kzxz{\tau}{0}{0}{-\overline{\zeta}},\overline{s})\rangle_{2} \im(\tau)^{\kappa}d\mu(\tau)\right)\frake_\lambda \\
    & = K(\kappa,s)Z(2s+\kappa,f)f(\zeta),
    \end{split}
\end{equation}
which holds  in the region of convergence of $E_{\kappa,0}^2$, i. e. for all $s$ with $\re(s) > \frac{3-\kappa}{2}$. Thus, $Z(s,f)$ converges in the  region of all $s$ with $\re(s)> \frac{3+\kappa}{2}$.
A more general zeta function is specified in \cite{St3}. It is defined as an Euler product
\[
\calZ(s,f) = \prod_{p \text{ prime}}\calZ_p(s,f),
\]
where for each prime $p$ the local zeta function  $\calZ_p(s,f)$ is given by
\[
\calZ_p(s,f) = \sum_{(k,l)\in \Lambda_+}\lambda_f(m(p^k,p^l))p^{-s(k+l)}.
\]
Here $\Lambda_+$ means the set
\[
\Lambda_+ = \left\{(k,l)\in \Z^2\; |\; 0\le k\le l\text{ and } k+l\in 2\Z \right\}. 
\]
The zeta function $\calZ(s,f)$ can be expressed in terms of $Z(s,f)$ by the following relation
\begin{equation}\label{eq:global_general_zeta_func}
  \begin{split}
    &\calZ(s,f) \\
    &=\prod_{p\mid |A|}\left(\left(\frac{e(\sig(A_p)/8)}{|A_p|^{1/2}} -1\right) + L_p(2s+\kappa-2,\chi_{A_p^\perp})\right)L(2s+\kappa-2,\chi_A)Z(s+\kappa-2,f),
    \end{split}
  \end{equation}
  where
  \begin{enumerate}
  \item[i)]
    $\chi_A$ and $\chi_{A_p^\perp}$ are Dirichlet characters, which are defined in \cite{St3}, Section 3,
  \item[ii)]
      \[
      L_p(s,\chi_{A_p^\perp})  =(1-\chi_{A_p^\perp}(p)p^{-s})^{-1},
      \]
    \item[iii)]
      $L(s,\chi_A)$  is the Dirichlet $L$-series associated to $\chi_A$.
  \end{enumerate}
  From \eqref{eq:global_general_zeta_func} we can immediately conclude that $\calZ(s,f)$ converges for $s\in \C$ with $\re(s)> \frac{7-\kappa}{2}$.

  In  \cite{St3}  an isomorphism between $S_{\kappa,A}$ and a space of vector valued automorphic forms $A_\kappa(\omega_f)$ of type $\omega$ is established. Vector valued spherical Hecke algebras depending on a representation are well studied objects. In \cite{St3} for each prime $p$ the structure of a subalgebra of the  vector valued spherical Hecke algebra $\calH(\calQ_p//\calK_p,\omega_p)$ depending on the Weil representation $\omega_p$ is determined. 
  Two cases have to be considered. In both cases we let
  \[
  \left\{T_{k,l}\; |\; (k,l)\in \Lambda_+\right\}
  \]
  be a set of generators.
  If $p$ divides $|A|$, this algebra is denoted with  $\calH^+(\calQ_p//\calK_p,\omega_p)$ and 
    \begin{align*}
    T_{k,l}(k_1m(p^r,p^s)k_2) = \omega_p(k_1)\circ T(k,l)\circ \omega_p(k_2)      
  \end{align*}
  with
  \[
  T(k,l)(m(p^r,p^s))\varphi_p^{(\mu_p)} = \frac{g(A_p)}{g_{p^l}(A_p)}\mathbbm{1}_{\calK_pm(p^k,p^l)\calK_p}\varphi_p^{(p^{(l-k)/2}\mu_p)} = \frac{g(A_p)}{g_{p^l}(A_p)}\mathbbm{1}_{\calK_pm(p^k,p^l)\calK_p}\varphi_p^{(0)}
  \]
  for $k< l$  and
  \[
  T(k,k)(m(p^r,p^s)) = \frac{g(A_p)}{g_{p^k}(A_p)}\mathbbm{1}_{\calK_pm(p^k,p^k)\calK_p}\id_{S_{L_p}}
  \]
  for $k=l$.
  If $p\nmid |A|$, the  whole spherical Hecke algebra $\calH(\calQ_p//\calK_p,\omega_p)$ is considered. As $\omega_p$ acts trivially in this case, its structure is much easier and well known. In this case we have
  \[
  T_{k,l} = \mathbbm{1}_{\calK_p m(p^k,p^l)\calK_p}\id_{S_{L_p}}. 
  \]
Here $g_{p^k}(A_p)$ (and $g(A_p)$) is given by \eqref{eq:Gauss_sum}.  
  Each local Hecke algebra acts on $A_\kappa(\omega_f)$ by a Hecke operator $\calT^{T_{k,l}}$. It is important to realize that this action is compatible with the action of Hecke operators on $S_{\kappa,A}$. More specifically, the identity
   \begin{equation}\label{eq;compatibility_notcoprime}
    \calT^{T_{k,l}}(F_f) = F_{p^{(k+l)(\frac{\kappa}{2}-1)}T(m(p^{-k},p^{-l}))f}
   \end{equation}
   holds, where $F_f\in A_\kappa(\omega_f)$ corresponds to $f\in S_{\kappa,A}$ and $T(m(p^{-k},p^{-l}))$ means the Hecke operator attached to $m(p^{-k},p^{-l})$ (cf. \cite{St3}, Thm. 6.9). 
   If $F$ is an automorphic eigenform  of all Hecke operators $\calT^{T_{k,l}}$, the corresponding eigenvalues give rise to an algebra homomorphism
\[
  \lambda_{F,p}:  \calH^+(\calQ_p//\calK_p,\omega_p)\rightarrow \C,\quad T\mapsto \lambda_{F,p}(T). 
  \]
  This algebra homomorphism determines in turn an unramified character $\chi_{F,p}$ on the group  $\calM_p$.
   The standard $L$-function $L(s,F)$ of a common eigenform $F$ is constructed as an Euler product
   \begin{equation}
    L(s,F) = \prod_{p<\infty}L_p(s,F)
  \end{equation}
  with
  \begin{equation}\label{eq:local_L_function}
    L_p(s,F)=
    \begin{cases}
      \frac{1+\chi_{F,p}^{(1)}(p)\chi_{F,p}^{(2)}(p)p^{-2s+1}}{(1-\chi_{F,p}^{(1)}(p^2)p^{-2s+1})(1-\chi_{F,p}^{(2)}(p^2)p^{-2s+1})}, & (p,|A|)=1,\\
      &\\
      \frac{1}{C(A_p)}\frac{1+\chi_{F,p}^{(1)}(p)\chi_{F,p}^{(2)}(p)p^{-2s+1}}{(1-\chi_{F,p}^{(1)}(p^2)p^{-2s+1})(1-\chi_{F,p}^{(2)}(p^2)p^{-2s+1})}, & p\mid |A|.
      \end{cases}
  \end{equation}
  Here $\chi_{F,p} = (\chi_{F,p}^{(1)}, \chi_{F,p}^{(2)})$ is the unramified character mentioned before and $C(A_p)$ is some constant depending on the $p$-group $A_p$ (see \cite{St3}, Lemma 7.3). 
  If $F_f$ is the automorphic form belonging to $f\in S_{\kappa,A}$ and $L(s,F_f)$ the standard $L$-function of $F_f$, we define the standard $L$-function  of $f$ naturally by
  \[
  L(s,f) = L(s,F_f). 
  \]
  There is a relation between the standard zeta function $\calZ(s,f)$ and the standard $L$-funktion $L(s,f)$. We have
  \begin{equation}\label{eq:rel_zeta_standard_L_func}
    \calZ(s+\frac{\kappa}{2}-1, f) = L(s,f). 
    \end{equation}
As $\calZ(s,f)$ converges for all $s\in \C$ with $\re(s)> \frac{7-\kappa}{2}$, it follows from \eqref{eq:rel_zeta_standard_L_func} that $L(s,f)$ converges for all $s\in \C$ with $\re(s) > \frac{9}{2}-\kappa$. 

\begin{theorem}\label{thm:L_func_non_zero}
Let $\kappa= \frac{m}{2}+l$ with $m$ and $l$ as in Sections \ref{sec:preliminaries} and \ref{subsec:schwarz_forms} satisfying $\frac{m}{2} > l + 3$.   Then the standard $L$-function $L(s,f)$ is non-zero for $s= -\frac{m}{4}-\frac{3}{2}l+3$.     \end{theorem}
\begin{proof}
  The condition $\frac{m}{2} > l + 3$ is equivalent to $\frac{m}{2} > \frac{3}{2}+\frac{\kappa}{2}$. Therefore, $\frac{m}{2}$ lies inside the region of convergence of  $Z(s,f)$. The identity \eqref{eq:special_values_zeta_standard_L} implies immediately that $L(s,f)$ converges for $s= -\frac{m}{4}-\frac{3}{2}l+3$. Alternatively, this can be confirmed directly by checking that
\[
3 -\frac{m}{4}-\frac{3l}{2}  > \frac{9}{2}-\frac{m}{2} - l.
\]
  $L(s,f)$ is defined as an Euler product. Thus, it suffices to prove that each factor of this product is non-zero at the point $s$ in question. In view of  \eqref{eq:local_L_function} this boils down to show that
  \begin{equation}\label{eq:nominator_L_function}
1 + \chi_{F,p}(m(p,p))p^{\frac{m}{2}+3l-5} =  1+\chi_{F,p}^{(1)}(p)\chi_{F,p}^{(2)}(p)p^{\frac{m}{2}+3l-5}
  \end{equation}
  is non-zero. For this we need to calculate the character $\chi_{F,p}$ evaluated at $m(p,p)$. It is determined by the algebra homomorphism $\lambda_{F,p}$ via the relation
\begin{equation}\label{eq:eigenvalue_char}
  \begin{split}
  \lambda_{F,p}(T_{k,l}) &=
  \begin{cases}
    \sum_{(r,s)\in \Z^2} S(\langle T_{k,l}, \varphi_p^{(0)}\rangle)(m(p^r,p^s))\chi_{F,p}(m(p^r,p^s)), &  (p, |A|)=1,\\
      \sum_{(r,s)\in \Z^2} \langle \calS(T_{k,l})(m(p^r,p^s)), \varphi_p^{(0)}\rangle\chi_{F,p}(m(p^r,p^s)), & p\mid |A|, 
  \end{cases}\\
  \end{split}
  \end{equation}
cf. equation  (7.21) in \cite{St3}. Here $S$ means the classical Satake map and $\calS$ the Satake map introduced in \cite{St3}, (5.12). These equations simplify significantly if we choose $T_{k,l} = T_{1,1}$. We find immediately from the  calculations in the proof of Theorem 5.11 of \cite{St3} that
\begin{equation}\label{eq:satake_non_coprime}
  \begin{split}
  (\calS T_{k,k}(m(p^k,p^k)) &= T_{k,k}(m(p^k,p^k))_{|S_{L_p}^{N(\Z_p)}}\\
  &=\frac{g(A_p)}{g_{p^k}(A_p)}\mathbbm{1}_{\calD_pm(p^k,p^k)\calD_p}\id_{S_{L_p}^{N(\Z_p)}}.
\end{split}
  \end{equation}
The same arguments lead in the case of $p\nmid |A|$ to
\begin{equation}\label{eq: satake_coprime}
(S T_{k,k}(m(p^k,p^k)) = \mathbbm{1}_{\calD_pm(p^k,p^k)\calD_p}\id_{S_{L_p}^{N(\Z_p)}}. 
\end{equation}
Replacing the Satake maps in \eqref{eq:eigenvalue_char} in both cases with the right-hand side of \eqref{eq:satake_non_coprime} and \eqref{eq: satake_coprime}, we obtain
\begin{equation}\label{eq:eigenvalue_satake_map}
  \lambda_{F,p}(T_{1,1}) =
  \begin{cases}
    \frac{g(A_p)}{g_{p}(A_p)}\chi_{F,p}(m(p,p)),  &  p\mid |A|,\\
    \chi_{F,p}(m(p,p)), & (p, |A|)=1.
    \end{cases}
\end{equation}
Combining the equations (3.14) and (3.15) of \cite{St3} with \eqref{eq;compatibility_notcoprime}, yields a relation between the eigenvalues $\lambda_{F.p}(T_{k,l})$ and $\lambda_f(T(m(p^{l-k},1))$:
\begin{equation}\label{eq:eigenvalue_relation}
  \begin{split}
  \lambda_{F,p}(T_{k,l}) &=
  \begin{cases}
    p^{(k-l)(\kappa/2-1)}\frac{g_{p^k}(A)}{g(A)}\frac{g(A)}{g_{p^{k+l}}(A)}\lambda_f(m(p^{l-k},1)), & p\mid |A|,\\
    p^{(k-l)(\kappa/2-1)}\lambda_f(m(p^{l-k},1)), & p\nmid |A|.
  \end{cases} \\
  \end{split}
\end{equation}
In the case $p\mid |A|$ we would like to further simplify the expression on the right-hand side of \eqref{eq:eigenvalue_relation}. First we keep in mind that $p^{k+l}$ is a square as $(k,l)$ is an element of $\Lambda_+$. This implies that $g_{p^{k+l}}(A)= g(A)$.  To evaluate the fraction $\frac{g_{p^k}(A)}{g(A)}$,  we decompose $A$ in the following way
\[
A = A_p \oplus A_p^{\perp},
\]
where $A_p^{\perp}$ is the orthogonal complement of the $p$-group $A_p$ in $A$. 
Following \cite{St3}, p. 10 ff, we have for any $r\in\N$
\begin{align*}
  \frac{g(A)}{g_{p^r}(A)} 
  &= \frac{g(A_p)}{g_{p^r}(A_p)}\frac{g(A_p^{\perp})}{g_{p^r}(A_p^\perp)} \\
  & = \frac{e(\sig(A_p)/8)}{|A_p|^{1/2}}\chi_{A_p^\perp}(p^r), 
\end{align*}
where $\chi_{A_p^\perp}$ is the quadratic character $n \mapsto \leg{n}{|A_p^\perp|}$. For the evaluation of of $\frac{g(A_p)}{g_{p^r}(A_p)}$ we have used Milgram's formula and the fact that $g_{p^r}(A_p) = |A_p|$.
Taking this into account,  we obtain in the case $p\mid |A|$ 
\begin{equation}\label{eq:eigenvalue_relation_simplified}
\lambda_{F,p}(T_{k,l}) = p^{(k-l)(\kappa/2-1)}\frac{g_{p^k}(A_p)}{g(A_p)}\chi_{A_p^\perp}(p^k)\lambda_f(m(p^{l-k},1)).
\end{equation}

Combining the equations \eqref{eq:eigenvalue_satake_map} and \eqref{eq:eigenvalue_relation_simplified}, we find 
\begin{equation}\label{eq:unramified_char}
  \begin{split}
  \chi_{F,p}(m(p,p)) &= 
  \begin{cases}
       \left(\frac{g_{p}(A_p)}{g(A_p)}\right)^2\chi_{A_p^\perp}(p)\lambda_f(m(1,1)),&  p\mid |A|, \\
      \chi_{A_p^\perp}(p)\lambda_f(m(1,1)), & (p, |A|)=1
  \end{cases} 
  \\
  & = \begin{cases}
    |A_p|e(-\sig(A_p)/4)\chi_{A_p^\perp}(p), &  p\mid |A|, \\
    \chi_{A_p^\perp}(p),  & (p, |A|)=1.
    \end{cases}
  \end{split}
\end{equation}
Replacing  $\chi_{F,p}(m(p,p))$ in \eqref{eq:nominator_L_function}  with the right-hand side of \eqref{eq:unramified_char} gives
\begin{equation}\label{eq:nominator_l_func_explicit}
  1 + \chi_{F,p}(m(p,p))p^{\frac{m}{2}+\frac{3}{2}l-5} =
  \begin{cases}
    1 + |A_p|e(-\sig(A_p)/4)\chi_{A_p^\perp}(p)p^{\frac{m}{2}+3l-5}, &  p\mid |A|, \\
    1 + \chi_{A_p^\perp}(p)p^{\frac{m}{2}+3l-5}, & (p, |A|)=1.
  \end{cases}
\end{equation}
  Note that $A_p$ is anisotropic and $\frac{m}{2} > l + 3$. Thus, as $|A_p|\ge p$, we have the estimate
  \[
  |A_p|p^{\frac{m}{2}+3l-5} \ge p
  \]
  and we can easily conclude that $1 + \chi_{F,p}(m(p,p))p^{\frac{m}{2}+\frac{3}{2}l-5}$ is non-zero in both of the above treated cases.
\end{proof}

\section{Injectivity of the Kudla-Millson theta lift}\label{sec:inj_kudla_millson}
In this section, we generalize the results in \cite{BF}, Chap. 4, to cusp forms of type $\rho_A$ where $L$ is not unimodular. We follow quite closely the steps of the proof in \cite{BF}, which carry over with some modifications to the general case. To this end, we  keep the notation of Section \ref{subsec:schwarz_forms} and return to the adelic setup of Section \ref{subsec:siegel_weil} as we want to make use of the adelic version of the Siegel-Weil formula.
Throughout this section we suppose that  the weight $\kappa$ is fixed and equal to
\[
\kappa = \frac{m}{2}+l.
\]

Let $C(V)$ be the Clifford algebra of the quadratic space $V$. It splits into a direct sum
\[
C(V)=C^+(V)\oplus C^{-}(V),
\]
where $C^+(V)$ is the subalgebra of even elements of $C(V)$.
We write $C^+(V)^\times$ for the invertible in $C^+(V)$. 
The general spin group $H=\GSpin(V)$ is defined by
\[
\GSpin(V) =\left\{g\in C^+(V)^\times\; \big\vert\; gVg^{-1}=V\right\}, 
\]
One can show that the action $\alpha$ of $H$ on $V$ given by
\begin{equation}\label{eq:act_gspin}
g\mapsto \alpha(g),\quad \alpha(g)(v)=gvg^{-1} 
\end{equation}
leaves the quadratic form $Q$ invariant. 
In fact, the group $H$ is connected to $\SO(V)$ by the following exact sequence
\[
1\longrightarrow \G_m\longrightarrow H\xlongrightarrow{\alpha} \SO(V)\longrightarrow 1.
\]
Note that the same construction generalizes to lattices $L$ in $V$  with the inclusion  $C(L) \subset C(V)$ and $\GSpin(L)\subset \GSpin(V)$. A good reference for further details is \cite{AD}.
Let $K_f^H = \prod_p K_p^H$ be an open compact subgroup of $H(\A_f)$ which leaves $L$ invariant and acts trivially on $A$ (see Remark \ref{rem:action_H_A_f} for for the action on $L$ and $L'/L$). To lighten the notation, we write  $K$ instead of $K_f^H$. Then there is a Shimura variety $X_K$ over $\Q$ associated to the Shimura datum $(D, H)$ whose $\C$-points are of the form  
\begin{equation}\label{eq:shimura_variety}
X_{K}(\C)  = H(\Q)\bs (D\times H(\A_f))/K.
\end{equation}
We identify $X_K$ with $X_K(\C)$.
It is well known (see \cite{Mi}, Lemma 5.13, or \cite{Ku1},p. 44-45) that the Shimura variety $X_K$ allows a finite decomposition into connected components. To describe these components, we note that by the strong approximation theorem one has 
\begin{equation}\label{eq:strong_approximation}
H(\A_f) = \bigsqcup_i H(\Q)^+ h_i K,
\end{equation}
with $h_i\in H(\A_f)$, where $H(\Q)^+ = H(\R)^+\cap H(\Q)$ and $H(\R)^+$ is the component of the identity  of $H(\R)$. Then
\begin{equation}\label{eq:shimura_variety_decomp}
X_K \cong \bigsqcup_i \Gamma_{i}\bs D^+
\end{equation}
with $\Gamma_i = H(\Q)^+\cap h_iKh_i^{-1}$ being a congruence subgroup of $H(\Q)^+$.
Throughout the rest of the paper we assume that the image of $H(\Q)^+\cap K$ in $SO^+(V)(\A_f)$ is isomorphic to a subgroup of finite index of the discriminant kernel $\Gamma(L)$ (see Remark \ref{rem:rel_adelic_classic_orth_modforms}). 
It is well known that $X_K$ has the analytic structure of a complex orbifold and is  a complex manifold if $K$ is neat. The same holds for the locally symmetric spaces $\Gamma_i\bs D^+$.
The isomorphism \eqref{eq:shimura_variety_decomp} yields
\begin{equation}\label{eq:differential_forms_x_k}
\calA^q(X_K) \cong \left[\calA^q(D)\otimes C^\infty(H(\A_f))\right]^{H(\Q)\times K} \cong \bigoplus_i \calA^q(D^+)^{\Gamma_i},
\end{equation}
where the second isomorphism is obtained by mapping a differential form $\eta(z,h)$ to the vector $(\eta(z,h_i))_{i}$, see \cite{Ku1}, p. 69.

\begin{remark}\label{rem:action_H_A_f}
  \begin{enumerate}
  \item[i)]
    The action of $H(\A_f)$ on $L$ and $L'$ via $\alpha$ in \eqref{eq:act_gspin} is understood as action on $L=\cap_p \left(V(\Q)\cap L_p\right)=V(\Q)\cap \hat{L}$ and $L'=\cap_p\left(V(\Q)\cap L'_p\right)$, respectively.
    We write
    \[
    L^h = \alpha(h)(L).
    \]
    As $(\alpha(h)(L))' = \alpha(h)(L')$, the action of $h\in H(\A_f)$ on $L$ and $L'$ induces an action on the discriminant group $A$. We obviously have  $A\cong (L^h)'/L^h$ and write $A^h$ for $(L^h)'/L^h$ and $h\mu=\alpha(h)(\mu)$ for any $\mu\in A$. As is pointed out in \cite{HMP}, the map
  \[
  h\cdot \sum_{\mu\in A}c_\mu \frake_\mu = \sum_{\mu\in A}c_\mu \frake_{h\mu}
  \]
  defines an isomorphism $\C[A]\cong \C[A^h]$  and isomorphic representations $\rho_A$ and $\rho_{A^h}$ for each $h\in H(\A_f)$. Consequently,
  the spaces $H_{\kappa,A}$ and $H_{\kappa,A^h}$ of weak Maass forms are isomorphic (by the map $\displaystyle f\mapsto h\cdot f$) and clearly the same is true for  all specified subspaces in Section \ref{subsec:weak_maas_forms}. We write $f^h$ for $h\cdot f$.

  We will later make use of the following fact, which can be found in \cite{Ho}, Lemma (10.2.8): If the image of $\Gamma_K:=\Gamma_1=H(\Q)^+\cap K$ is exactly $\Gamma(L)$, then the image of $\Gamma_i=H(\Q)^+\cap h_iKh_i^{-1}$ is given by $\Gamma(L^{h_i})$. 
\item[ii)]
Let  $\varphi_{q,l}$ be the Schwartz form \eqref{eq:funke_millson_form} and $\varphi_\mu, \; \mu\in A$ as in \eqref{eq:familiy}.   Attached to $\varphi_{q,l}$ and $\varphi_\mu$, we define a vector valued Siegel theta function on $D\times H(\A_f)$.
  Similar to \eqref{eq:theta_function_adelic} we put
  \begin{align*}
    \vartheta((g_\tau,1_f),(z,h),\varphi_{q,l}\otimes\varphi_\mu) &= \sum_{x\in V(\Q)}\varphi_\mu(h^{-1}x)\omega((g_\tau,1_f))\varphi_{q,l}(x,z), 
  \end{align*}
  where $h\in H(\A_f)$ and $g_\tau = n(u)m(\sqrt{v})\in G'(1)$ moves $i$ to $\tau=u+iv\in \H$.  Analogous to  \eqref{eq:theta_gaussian}, we subsequently set
  \begin{equation}\label{eq:theta_schwartz_form}
    \begin{split}
      \Theta_A(\tau,(z,h),\varphi_{q,l}) &= \sum_{\mu\in A}\vartheta((g_\tau,1_f),(z,h),\varphi_{q,l}\otimes\varphi_\mu)\varphi_\mu\\
      &=v^{-\kappa/2}\sum_{\mu\in A}\sum_{\lambda\in h(\mu+L)}\varphi_{q,l}(\sqrt{v}\lambda,z)e^{\pi i(\lambda,\lambda)u}\varphi_\mu,
      \end{split}
  \end{equation}
 where $h(\mu+L)$ is meant in sense of part i) of this remark. We sometimes use the notation $\Theta_\mu(\tau,(z,h),\varphi_{q,l})$ for the $\mu$-th component of $\Theta_A$. 

\end{enumerate}
\end{remark}

Note that $\Theta_A$  descends to a $q$-form on $X_K$ since $K$ leaves $A$ invariant and stabilizes $L$. In this case $\varphi_\mu$ is also  invariant under the action of $K$. 
Based on  Theorem \ref{thm:properties_varphi_q_l}, the usual arguments then yield the following important properties of $\Theta_A(\tau,(z,h),\varphi_{q,l})$ (cf. \cite{FM}, Prop. 7.1, and \cite{Ku2}, p. 301). 
\begin{theorem}\label{thm:theta_prop}
  Let $K\subset H(\A_f)$ be as above. Then $\Theta_A(\tau,(z,h),\varphi_{q,l})$ defines a $\Sym^l(V)$-valued closed $q$-form on the Shimura variety $X_K$. Also, as a function  on $\H$ it is a non-holomorphic vector-valued modular form of weight $\kappa$ transforming according to $\rho_A$. 
\end{theorem}
In view of this theorem and the fact that $\vartheta((g_\tau,1_f), h, \varphi_{q,l}(z)\otimes \varphi_\mu)$ is slowly increasing in $\tau$ (see \cite{Ku2}, p. 324), the following definition makes sense.

\begin{definition}\label{def:kudla_millson_theta_lift}
Let $f = \sum_{\lambda\in L'/L}f_\lambda\frake_\lambda\in S_{\kappa,A}$ be a cusp form.
Then $f\mapsto \Lambda(f)$ with
\begin{equation}\label{eq:kudla_millson_theta_lift}
  \Lambda(f)(z,h):=\int_{\Gamma\bs \H}\langle f(\tau), \Theta_A(\tau,(z,h),\varphi_{q,l})\rangle \im(\tau)^\kappa d\mu(\tau)
\end{equation}
defines a linear map
 \[
\Lambda: S_{\kappa,A}\longrightarrow \calZ^q(X_K,\widetilde{\Sym}^l(V))
\]
Here $\widetilde{\Sym}^l(V)$ is the local system on $D$ associated to $\Sym^l(V)$. 

Following \cite{BF}, the $L^2$-norm of $\Lambda$ is defined as
\begin{equation}\label{eq:l2_norm}
\|\Lambda(f)\|_2^2 = \int_{X_{K}}\Lambda(f)\wedge \ast\overline{\Lambda(f)}.
\end{equation}
\end{definition}

The subsequent proposition ensures that $\Lambda(f)$ is square integrable. It follows from the scalar valued companion statement in Prop. 4.1, \cite{BF}. To phrase this result, we introduce the following notation:

In accordance with \eqref{eq:theta_schwartz_form} we write 
\begin{equation}\label{eq:theta_schwarz_form_A_A}
  \begin{split}
    &  \Theta_{A^2}(\tau_1,\tau_2,(z,h),\phi_{q,l})\\
    &= (v_1v_2)^{-\kappa/2}\sum_{(\lambda,\nu)\in A^2}\left(\sum_{{\bf x}\in V^2(\Q)}\varphi_{(\lambda,\mu)}(h^{-1}{\bf x})\omega_\infty(\iota(g_{\tau_1},g_{\tau_2}))\phi_{q,l}({\bf x},z)\right)\varphi_{(\lambda,\mu)},
  \end{split}
\end{equation}
where $\phi_{q,l}$ is the Schwartz function  in \eqref{eq:phi_q_l}, $\varphi_{(\lambda,\mu)} = \varphi_\lambda\otimes\varphi_\mu$ and $\iota$ is the standard embedding in \eqref{eq:embed_sl}. Finally, we set
\begin{equation}\label{eq:theta_integral}
  I(\tau_1,\tau_2,\phi_{q,l}) = \int_{X_K}\Theta_{A^2}(\tau_1,\tau_2,\phi_{q,l})\mu,
  \end{equation}
where $\mu$ is the volume form on $D$ specified in Section \ref{subsec:schwarz_forms}.   A similar theta integral is also studied in \cite{Ku2}. Given Proposition \ref{prop:theta_integrals_relation}, this integral exists if Weil's convergence criterion is fulfilled.  

\begin{proposition}\label{prop:l2_norm_theta_integral}
  Take the assumptions of Definition \ref{def:kudla_millson_theta_lift} and let $m > r_0 + 3$ such that Theorem \ref{thm:siegel_weil} holds. Then $\Lambda(f)$ is square integrable and
  \begin{equation}\label{eq:l2_norm_formula}
  \|\Lambda(f)\|_2^2 = \left(f(\tau_1)\otimes \overline{f(\tau_2)},I(\tau_1,-\overline{\tau}_2, \phi_{q,l})\right),
  \end{equation}
  where $(\cdot,\cdot)$ is the Petersson scalar product on $S_{\kappa,A}\otimes S_{\kappa,A}$.
\end{proposition}
\begin{proof}
  As in Prop. 4.1, \cite{BF}, we argue that $\|\Lambda(f)\|_2^2$ indeed exists if \eqref{eq:l2_norm_formula} holds. As proved in \cite{Ku2}, Theorem 3.1, in a more general and complicated situation and also done in the scalar valued case in Prop. 4.1  of \cite{BF},  we may exchange the order of the integrals over $X_K$ and $\Gamma\bs \H$. Following the remaining proof of Prop. 4.1, for any pair $\lambda,\nu\in A$ we also have
  \[
  \Theta_\lambda(\tau_1,\varphi_{q,l}\otimes\varphi_\lambda)\wedge \overline{\Theta_\nu(\tau_2,\ast\varphi_{q,l}\otimes\varphi_\nu)} = \Theta_{(\lambda,\nu)}(\tau_1,-\overline{\tau_2},\phi_{q,l}\otimes(\varphi_\lambda\otimes\varphi_\nu))\mu,
  \]
  where $\Theta_{(\lambda,\nu)}$ is the component of  $\Theta_{A^2}$  belonging to the index $(\lambda,\mu)\in A^2$.
  In view of \eqref{eq:tensor_product_forms} we obtain the assertion.
\end{proof}

In the next result we want to replace the theta kernel $\phi_{q,l}$ in  the integral $I$ with the Schwartz function $\xi$ in \eqref{eq:xi_q_l} with the help of the relation \eqref{eq:rel_xi_phi_q_l}. To this end, we have to interpret $\xi$ as an element of $\displaystyle \left[S(V(\R)^2)\otimes C^\infty(D)\right]^{G(\R)}$. In view of Lemma \ref{lem:properties_xi}, $i)$, we may set 
\begin{equation}\label{eq:xi_on_D}
  \xi({\bf x}, z) = \xi(g^{-1} {\bf x})
  \end{equation}
with $g \in G(\R)$ such that  $g z_0 = z$, where $z_0\in D$ is a fixed  base point.
 

\begin{proposition}\label{prop:l2_norm}
  Let $\xi$ be the Schwartz function in \eqref{eq:xi_on_D}, $\theta_{A^2}(\tau_1,\tau_2, (z,h), \xi)$ and $I(\tau_1,\tau_2,z,\xi)$ as in \eqref{eq:theta_schwarz_form_A_A} and \eqref{eq:l2_norm_formula}, respectively with $\phi_{q,l}$ replaced by $\xi$. Then
  \begin{equation}\label{eq:l2_norm_formula_1}
    \|\Lambda(f)\|_2^2 = \left(f(\tau_1)\otimes \overline{f(\tau_2)},I(\tau_1,-\overline{\tau_2}, \xi)\right).
    \end{equation}
  Moreover, $\Lambda$ vanishes identically if $p=1$ and $q+l>1$.
\end{proposition}
\begin{proof}
  The result is an immediate consequence  from Prop. 4.3 and Corollary 4.4  in \cite{BF} combined with \eqref{eq:l2_norm_formula}. The second assertion is due to Lemma \ref{lem:properties_xi}, $iii)$. 
  \end{proof}

The next proposition shows that the theta integral $I$ over $X_K$ can be written in terms of the theta integral in the Siegel-Weil formula in \eqref{eq:l2_norm_formula_1}. This justifies the convergence of $I$, which was required in Prop. \ref{prop:l2_norm_theta_integral}. The analogous statement in a scalar valued setting can be found in \cite{BF}, Prop. 4.6, \cite{Ku2}, Prop. 4.17.  These  papers  presume that $X_K$ is a manifold and  require the condition
\begin{equation}\label{eq:center_K}
  Z(\A_f)\cap K = \widehat{\Z}^\times
  \end{equation}
  to be satisfied, where $Z(\A_f)$ means the center of $H(\A_f)$. However, the proof of Prop. 4.17 in \cite{Ku2} should still work if $X_K$ is an orbifold. But for our purposes it sufficient for $X_K$ to be a complex manifold.  


\begin{proposition}\label{prop:theta_integrals_relation}
Let $m > r_0 + 3$ such that Theorem \ref{thm:siegel_weil} holds. Suppose further that the image  $\alpha(H(\Q)^+\cap K)$ in $SO(V)(\A_f)$ is isomorphic to the discriminant kernel and that $Z(\A_f)\cap K\cong \widehat{\Z}^\times$. Then 
  \begin{equation}\label{eq:theta_integrals_relation}
    \frac{1}{\vol(X_K,\mu)}I(\tau_1,\tau_2,\xi) = (v_1v_2)^{-\kappa/2}\int_{O(V)(\Q)\bs O(V)(\A)}\vartheta_L(\iota(g_{\tau_1},g_{\tau_2}),h,\xi)dh.
    \end{equation}
\end{proposition}
\begin{proof}
By Remark \ref{rem:action_H_A_f}, i), $\alpha(\Gamma_i)\cong \Gamma(L^{h_i})$ for all $i$. It is stated in \cite{Br1}, p. 115, that $\Gamma(L^{h_i})/D^+$ is Riemannian manifold, which implies that $X_K$ is a complex manifold. Applied to each component on both sides of \ref{eq:theta_integrals_relation}, Prop. 4.6 in \cite{BF1} and Prop. 4.17 in \cite{Ku2} yields the claimed assertion. 
\end{proof}

The Siegel-Weil formula \ref{prop:siegel_weil_vec_val} combined with Proposition \ref{prop:theta_integrals_relation} allows us to express the $L^2$-norm of $\Lambda(f)$ as a Rankin-Selberg type integral. The doubling method for our setup then leads to  a formula for $\|\Lambda(f)\|_2^2$  in terms of a special value of the standard zeta function associated to a common Hecke eigenform $f$. 
For the next theorem we use the following notation
\[
K(A_p,m,l) = \prod_{p\mid |A|}\left(\left(\frac{e(\sig(A_p)/8)}{|A_p|^{1/2}} -1\right) + L_p\left(\frac{m}{2}-l+2,\chi_{A_p^\perp}\right)\right)^{-1}
\]
(see Section \ref{sec:standard_l_func}).
Under the assumption  $\frac{m}{2}> l-1$ by taking into account that $\chi_{A_p^\perp}$ is a quadratic Dirichlet character, we find that  
\[
\left(\frac{e(\sig(A_p)/8)}{|A_p|^{1/2}} -1\right) + L_p\left(\frac{m}{2}-l+2,\chi_{A_p^\perp}\right) \not=0
\]
for each $p$ in the above product. 

\begin{theorem}\label{thm:l2_norm_doubling_int}
  Let $m, q, l$ be as before with $m > \max(6,2l-2,3+r_0)$ and $s_0 = (m-3)/2$. Furthermore, we assume that the conditions of Proposition \ref{prop:theta_integrals_relation} are satisfied and that $q+l$, $\kappa = \frac{m}{2}+l$ are even and $A$ is an anisotropic quadratic module. If  $f\in S_{\kappa,A}$ is a common eigenform of all Hecke operators $T\kzxz{d^2}{0}{0}{1}$, we have
  \begin{equation}\label{eq:l2_norm_special_values}
    \begin{split}
       \frac{1}{\vol(X_K,\mu)}\frac{\|\Lambda(f)\|_2^2}{\|f\|_2^2} &= C(s_0)K(\kappa,-l/2)L\left(\frac{m}{2}-l+2,\chi_A\right)^{-1}\times \\
      &\times\prod_{p\mid |A|}K(A_p,m,l)L(-\frac{m}{4}-\frac{3l}{2}+3,f),
      \end{split}
    \end{equation}
  where
  \begin{equation}\label{eq:konstant_archimedian_place}
  K(\kappa,s) = \frac{e(\sig(A)/8)}{|A|^{1/2}}(-1)^{s+\frac{\kappa}{2}}2^{2-2s-\kappa+1}\frac{\Gamma(\kappa+ s-1)}{\Gamma(\kappa + s)}. 
  \end{equation}
\end{theorem}
\begin{proof}
  By the Propositions \ref{prop:l2_norm_theta_integral}, \ref{prop:theta_integrals_relation} and the Siegel-Weil formula in Corollary \ref{prop:siegel_weil_vec_val} we have
  \begin{equation}\label{eq:l2_norm_equation}
    \begin{split}
&    \frac{1}{\vol(X_K,\mu)}\|\Lambda(f)\|_2^2 = \left(f(\tau_1)\otimes \overline{f(\tau_2)}, \frac{1}{\vol(X_K,\mu)}I(\tau_1,-\tau_2,\xi)\right) \\
        &= \left(f(\tau_1)\otimes \overline{f(\tau_2)},(v_1v_2)^{-\kappa/2}\int_{O(V)(\Q)\bs O(V)(\A)}\vartheta_L(\iota(g_{\tau_1},g_{-\overline{\tau_2}}),h, \xi)dh\right) \\
        &= \left(f(\tau_1)\otimes \overline{f(\tau_2)},(v_1v_2)^{-\kappa/2}\sum_{(\lambda,\nu)\in A^2}E(\iota(g_{\tau_1},g_{-\overline{\tau_2}}),s_0,\Xi\otimes\Phi_{(\lambda,\nu)})\varphi_{(\lambda,\nu)}\right).
      \end{split}
    \end{equation}
  Bearing  Proposition \ref{prop:standard_section_xi} in mind, we see that
\[
(v_1v_2)^{-\kappa/2}\sum_{(\lambda,\nu)\in A^2}E(\iota(g_{\tau_1},g_{-\overline{\tau_2}}),s_0,\Xi\otimes\Phi_{(\lambda,\nu)})\varphi_{(\lambda,\nu)}
\]
is nothing else but the Eisenstein series of genus 2 defined in \cite{St1}, Def. 3.13. That being said, we may apply Lemma 3.14 of \cite{St1} and obtain for the right-hand side of \eqref{eq:l2_norm_equation}
  \begin{align*}
    C(s_0)\left(f(\tau_1)\otimes \overline{f(\tau_2)},E_{\kappa,0}^2(\kzxz{\tau_1}{0}{0}{-\overline{\tau_2}},-\frac{l}{2})\right).
  \end{align*}
  By means of \cite{St2}, Theorem 6.4, this becomes
  \begin{align*}
    C(s_0)K(\kappa,-l/2)Z(2(-\frac{l}{2})+\kappa, f)\int_{\Gamma\bs \H}\sum_{\lambda\in A}\overline{f_\lambda(\tau_2)}f_\lambda(\tau_2)\im(\tau_2)^{\kappa}d\mu(\tau_2).
  \end{align*}
  As $\frac{m}{2} > l-1$, it is a classical result that $L\left(\frac{m}{2}-l+2,\chi_A\right)\not=0$ and we may then express $Z(\frac{m}{2},f)$ in terms of $\calZ(-l+2,f)$ using the equation (3.24) of \cite{St3}.  Subsequently, employing (7.25) of \cite{St3}, yields
  \begin{equation}\label{eq:special_values_zeta_standard_L}
    \begin{split}
     Z(\frac{m}{2},f)&= \prod_{p\mid |A|}\left(\left(\frac{e(\sig(A_p)/8)}{|A_p|^{1/2}}-1\right)+L_p\left(\frac{m}{2}-l+2,\chi_{A_p^\perp}\right)\right)^{-1}\times \\
     & L\left(\frac{m}{2}-l+2,\chi_A\right)^{-1}L\left(-\frac{m}{4}-\frac{3}{2}l+3,f\right).
    \end{split}
    \end{equation}
\end{proof}

As a corollary we can deduce the injectivity of the lifting $\Lambda$.
\begin{corollary}\label{cor:injectivity_kudla_millson}
  Under the conditions of Theorems \ref{thm:l2_norm_doubling_int} the Kudla-Millson theta lift $\Lambda$ \eqref{eq:kudla_millson_theta_lift} is injective.
  \end{corollary}
\begin{proof}
  It suffices to prove that $\|\Lambda\|_2^2$ is non-zero. In view of \eqref{eq:l2_norm_special_values} and \eqref{eq:konstant_archimedian_place} we need to show that $L(-\frac{m}{4}-\frac{3l}{2}+3,f)$ is non-zero. But this is just the assertion of Theorem \ref{thm:L_func_non_zero}. 
  \end{proof}

\section{Surjectivity of the Borcherds lift}\label{sec:surj_borcherds}
In this section we pick up the question from the introduction whether a modular form for some orthogonal group with zeroes and poles located on Heegner divisors  can be realized as a Borcherds lift of weakly holomorphic modular forms. 
The most general results in this direction are given in \cite{Br2}. We focus here  on Theorem 1.4 in \cite{Br2}, which essentially only assumes that the level of the lattice $L$ is a prime number. In particular, it is not required that $L$ splits a lattice of the form $U\otimes U(N)$ over $\Z$. Our assertion is in the same vein. We do not impose any further restrictions on the lattice, but assume that the discriminant group $A=L'/L$ is anisotropic. Thus,
\[
A = \bigoplus_p A_p, 
\]
where each $p$-component $A_p$ is anisotropic (see \eqref{eq:anisotropic_finite_modules}).

Before stating our results, we briefly gather the necessary facts on modular forms on orthogonal groups and review in some detail the involved lifts in an adelic setting suited to our needs.  We follow loosely \cite{Br1}, \cite{Ku2}.  We adopt the notation from Section \ref{subsec:schwarz_forms} and \ref{sec:inj_kudla_millson}, but restrict ourselves to the Hermitian case  and assume that $l$ (the parameter of the Schwartz form $\varphi_{q,l}$) is zero. Accordingly,  $(V(\R),Q)$ is a quadratic space of type $(p,2)$ and $D$ is the Grassmannian of negative definite oriented subspaces $z\subset V(\R)$ of dimension 2. By $D^+$ we mean one of its two connected components. Let $G(\R)^+$ be the subgroup of $G(\R) = O(V)(\R)$ which preserves $D^+$ and $D^-$. It acts transitively on $D^+$.  Also, we define and write $X_K = H(\Q)\bs D\times H(\A_f)/K$ with the same meaning as in the section before. In particular, $K$ is an open compact subgroup of $H(\A_f)$ which preserves $L$ and acts trivially on $A$.
Furthermore, there are two weights involved in this section. On the one hand we reserve $\kappa$ for $\frac{m}{2} = 1 +\frac{p}{2}$, on the other hand we use $\ell = 2-\kappa = 1-\frac{p}{2}$. We stick with this notation throughout this section. \newline
To define an analogue of the upper half plane in the orthogonal setting, we give $D$ a complex structure. For this purpose, we consider the complexified space $V(\C) = V\otimes_\Q \C$  of $V$ and extend $(\cdot,\cdot)$ to a $\C$-bilinear form. Then
\begin{equation}\label{eq:cal_K}
  \calK = \left\{[z]\in P(V(\C))\; |\; (z,z)=0 \text{ and } (z,\overline{z})<0\right\}
  \end{equation}
is a complex manifold with two connected components, which are exchanged by $z\mapsto \overline{z}$. For $[z]\in \calK$ we utilize the notation  $z = x+ iy$. In terms  of this decomposition one can show that $[z]\mapsto \R x + \R y$  is a bijection between $\calK$ and $D$ inducing a complex structure on $D$.

We write $\calK^+$ for the component in $\calK$, which corresponds to $D^+$. Further, denote with $\widetilde{\calK}^+$ the preimage in $V(\C)$ of $\calK^+$ under the natural projection into the projective space $P(V(\C))$. The following definition is taken from \cite{Eh}, Def. 1.5.21.
\begin{definition}\label{def:orth_mod_form_D_A_f}
A function $F: \widetilde{\calK}^+\times H(\A_f)\rightarrow \C$ is called meromorphic modular form of weight $r\in \Z$, level $K\subset H(\A_f)$ and unitary character $\chi$ of finite order for $H(\Q)$ if
\begin{enumerate}
\item[i)]
  $z\mapsto f(z,h)$ is meromorphic for any fixed $h\in H(\A_f)$,
\item[ii)]
  $f(z,hk) = f(z,h)$ for all $k\in K$,
\item[iii)]
  $f(tz,h) = t^{-r}f(z,h)$ for all $t\in \C\bs \{0\}$,
\item[iv)]
  $f(\gamma z, \gamma h) = \chi(\gamma)f(z,h)$ for all $\gamma\in H(\Q)$ ,
\item[v)]
  and $f$ is meromorphic at the boundary components of $\widetilde{\calK}^+$.
\end{enumerate}
\end{definition}

\begin{remark}\label{rem:rel_adelic_classic_orth_modforms}
  Let $L$ be a lattice, which is additionally assumed to be a maximal lattice (that is there is no even lattice $M$ with $L\subsetneq M\subset V$) and set
  \begin{equation}\label{def:open_compact_GSpin}
  K= H(\A_f)\cap C^+(\hat{L})^\times,  
  \end{equation}
  where $L=L\otimes \hat{Z}$ as in Section \ref{subsec:siegel_weil} and $C^+(\hat{L})^\times$ is the unit group of the even Clifford algebra of the lattice $\hat{L}$.  With this choice Andreatta et. al showed that the $\C$-points of the GSpin-Shimura variety \eqref{eq:shimura_variety} is connected if $p\ge 2$ or the order of $A$ is square-free (see \cite{AGHM}, Prop. 4.1.1). In this case and in view of  \eqref{eq:shimura_variety_decomp}  we consequently have  
  \[
  X_K(\C) = \Gamma_K\bs D^+.
  \]
  It turns out that the group $\Gamma_K=H(\Q)^+\cap K$ can be specified explicitly: The map $\alpha$ restricted to $K$ defines a homomorphism $K\rightarrow \SO(\hat{L})$ whose exact image is the subgroup of elements acting trivially on the discriminant group $L'/L\cong \hat{L}'/\hat{L}$. Therefore $\alpha(\Gamma_K)$ can be identified with the discriminant kernel $\Gamma(L)$ as defined in \cite{Br1} or \cite{Br2}.
We may conclude that in this setting for  any meromorphic modular form $F$ in the sense of Definition \ref{def:orth_mod_form_D_A_f} the function $F(\cdot, 1)$ behaves like a meromorphic modular form of the same weight for the orthogonal group $\Gamma(L)$.
Finally, it is worth mentioning that owing to its definition, $K$ satisfies \eqref{eq:center_K}.
\end{remark}

Examples of modular forms of Definition \ref{def:orth_mod_form_D_A_f} can be constructed by means of the celebrated Borcherds lift. In its original form (\cite{B}, Theorem 13.3), it takes weakly holomorphic modular forms $f\in M_{\ell,A}^!$ of weight $\ell$  to meromorphic modular forms on orthogonal modular groups. Bruinier (\cite{Br1}) extended the lift to harmonic weak Maass forms $f\in H_{\ell,A}^+$  as a supply of inputs. Here we recall in line with Def. \ref{def:kudla_millson_theta_lift} the regularized theta  lift on $D\times H(\A)$. 
Its definition is  based on an integral of the form
\begin{equation}\label{eq:regularized_theta_int}
  \Phi_L(f)(z,h) = \int^{\reg}_{\Gamma\bs \H}\langle f(\tau),\Theta_A(\tau,(z,h),\varphi_\infty^{p,2})\rangle d\mu(\tau),
\end{equation}
where  $\Theta_A$ is defined in \eqref{eq:theta_schwartz_form} with $\varphi_{q,l}$ replaced with the Gaussian $\varphi_\infty^{p,2}$ (cf. \eqref{eq:gaussian}). 
Since $f$ grows exponentially for $\im(\tau)\rightarrow \infty$, the corresponding integral over $\Gamma\bs \H$ has to be regularized according to Borcherds \cite{B}, p. 514 ff. The regularized integral is denoted with $\int_{\Gamma\bs \H}^{\reg}$.
\begin{remark}\label{rem:adelic_theta_lift}
  Let $f\in H^+_{\ell,A}$ (or an element of $M_{\ell,A}!$) and $\Phi_L(f)(z,h)$ the regularized theta lift \eqref{eq:regularized_theta_int}. 
  Howard observed in \cite{HMP} that this adelic theta lift can be expressed as a classical regularized theta lift. More precisely, in terms of the notation of Remark \ref{rem:action_H_A_f} we have
  \begin{equation}\label{eq:adelic_classical_theta_rel}
    \begin{split}
      \Phi_L(f)(z,h) &= \Phi_{L^h}(f^h)(z)\\
      & = \int_{\Gamma\bs \H}^{\reg}\langle f^h(\tau),\Theta_{A^h}(\tau,z,\varphi_\infty^{p,2})\rangle d\mu(\tau), 
    \end{split}
  \end{equation}
  which can be easily confirmed with the help of \eqref{eq:theta_schwartz_form}.
  For the same reasons, the same holds for the Kudla-Millson theta lift, i. e.
  \begin{equation}\label{eq:adelic_classic_kudla_millson_theta_rel}
    \begin{split}
    \Lambda(f)(z,h) &=\Lambda(f^h)(z)\\
    &= \int_{\Gamma\bs \H}\langle f^h(\tau),\Theta_{A^h}(\tau,z,\varphi_{q,l})\rangle d\mu(\tau),
    \end{split}
  \end{equation}
  where the underlying lattice for the lift on the right-hand side  is $L^h$ as well.

\end{remark}



It is a fundamental theorem of Borcherds (\cite{B}, Thm. 6.2, or \cite{Br2}, Thm. 2.12) that $\Phi_L$ is  a smooth function on $X_K$  apart from logarithmic singularities along a certain divisor on $X_K$. This divisor, denoted with $Z(f)$,  is determined by the principal part of the lifted form $f$ and is a linear combination of so-called {\it Heegner divisors}, which we will now briefly describe. Following \cite{Ku1}, let $x\in V(\Q)$ with $Q(x)>0$ and $V_x=x^\perp$ the orthogonal complement of $x$ in $V(\Q)$. We write $H_x$ for the stabilizer of $x$ in $H$. As is noted in \cite{Ku2}, we have $H_x\cong \GSpin(V_x)$. Further, 
\begin{equation}\label{eq:D_x}
  D_x=\{z\in D\; |\; (z,x) = 0\}
  \end{equation}
defines an analytic set of codimension one in $D$. We put $D_x^+= D_x \cap D^+$.
For $h\in H(\A_f)$ we define a divisor $Z(x,h,K)$ on $X_K$ by the image of the map
\begin{equation}\label{eq:map_divisor}
  \begin{split}
    H_x(\Q)\bs D_x \times H_x(\A_f)/\left(H(\A_f)\cap hKh^{-1}\right)&\longrightarrow H(\Q)\bs D\times H(\A_f)/K,\\
   & (z,h_1)\mapsto (z,h_1h). 
    \end{split}
\end{equation}
It can be shown that $Z(x,h,K)$ is rational over $\Q$. Let $\mu\in A$. In terms of  an $n\in\Q_{>0}$ and the Schwartz function $\varphi_\mu$ (see \eqref{eq:familiy}) we introduce a weighted sum $Z(n,\varphi_\mu,K)$ of these divisors. To that end, we consider the set
\begin{equation}\label{eq:omega_n}
  \Omega_n = \left\{x\in V\;|\; Q(x)=n\right\}.
  \end{equation}
According to \cite{Ku2}, we may write for a fixed $x_0\in \Omega_n$ 
\begin{equation}
  \supp(\varphi_\mu)\cap \Omega_n(\A_f) = \bigsqcup_r K x_r^{-1}\cdot x_0
\end{equation}
with a finite set of elements $x_r\in H(\A_f)$.  We then define $Z(n,\varphi_\mu,K)$ by
\begin{equation}\label{eq:weighted_divisor}
  Z(n,\varphi_\mu,K)= \sum_{r}\varphi_\mu(x_r^{-1}\cdot x_0)Z(x_0,x_r,K)
\end{equation}
and $Z(n,\varphi_\mu,K)=0$ if $\Omega_n(\Q)$ is empty. Finally,
$Z(f)$ is given by
\begin{equation}\label{eq:divisor_weak_maasform}
  Z(f) = \frac{1}{2}\sum_{\mu\in A}\sum_{n>0}c^+(-n,\mu)Z(n,\varphi_\mu,K).
\end{equation}
Here $c^+(-n,h)$ are the Fourier coefficients of the principal part of $f$ (see \eqref{eq:principal_part_weak_maass}). 

With respect to \eqref{eq:shimura_variety_decomp} the divisor $Z(n,\varphi_\mu,K)$ can be decomposed into a finite sum of divisors $Z_i(n,\varphi_\mu,K)$, where $Z_i(n,\varphi_\mu,K)$ is a divisor on  the  connected component $\Gamma_i\bs D^+$:
\begin{equation}\label{eq:heegner_div_decomp}
  \begin{split}
   Z(n,\varphi_\mu,K) = \sum_i Z_i(n,\varphi_\mu,K) \text{ with } 
  Z_i(n,\varphi_\mu,K) = \sum_{x\in \Gamma_i\bs \Omega_n(\Q)}\varphi_\mu(h_i^{-1}x)\pr_i(D_x^+),
  \end{split}
\end{equation}
where $\pr_i$ maps $z\in D_x^+$ to $\Gamma_iz$ in $\Gamma_i\bs D^+$. Note that $Z_i(n,\varphi_\mu,K)$ can be written in a more explicit way by
\begin{equation}\label{eq:div_explicit}
 Z_i(n,\varphi_\mu,K) = \sum_{\substack{\lambda_i\in \Gamma_i\bs h_i\mu+L^{h_i}\\Q(\lambda_i) = n}}\pr_i(D_{\lambda_i}^+).
\end{equation}
Also, as $K$ leaves $L$ invariant and stabilizes $A$, the same holds for $\Gamma_i$ regarding $L^{h_i}$.  We use the notation $\mu_i = h_i\mu$.  Thus, the sum in \eqref{eq:div_explicit} is invariant under the action of $\Gamma_i$ and consequently
\begin{equation}\label{eq:div_components}
  \begin{split}
    Z_i(n,\varphi_\mu,K) &= \pr_i\left(\sum_{\substack{\lambda_i\in h_i\mu+L^{h_i}\\Q(\lambda_i) = n}}D_{\lambda_i}^+\right) \\
    &= \sum_{\substack{\lambda\in \mu_i +L^{h_i}\\Q(\lambda) = n}}D_{\lambda}^+. 
\end{split}
  \end{equation}
Note that the last expression in \eqref{eq:div_components} is just the Heegner divisor $H(\mu_i,n)$ as defined in \cite{Br1}.  
For simplicity, we write $Z(n,\mu)$ and $Z_i(n,\mu_i)$ instead of $Z(n,\varphi_\mu,K)$ and $Z(n,\varphi_\mu,K)$, respectively.





Now Theorem 13.3 in \cite{B}  can be transferred to our adelic setting (see \cite{Ku2}, Thm 1.3 or \cite{Eh}, Thm. 1.8.1), which reads as follows:
\begin{theorem}\label{thm:Borcherds_thm13_3}
  Let $f\in M^!_{\ell,L}$ with $c(\mu,n)\in \Z$ for all $m< 0$ and $c(\mu,n)\in \Q$ for all $n\in \Z+\Q(\mu)$. Then there is a function $\Psi_L(f)(z,h)$ on $D\times H(\A_f)$ such that
  \begin{enumerate}
  \item[i)]
    $\Psi_L(f)(z,h)$ is a meromorphic modular form of weight $c(0,0)/2$, level $K$ and some unitary character of finite order for $H(\Q)$.
  \item[ii)]
    the divisor of $\Psi_L(f)(z,h)$ on $X_K$ is given by $Z(f)$.
  \item[iii)]
    $\Psi_L$ is related to $\Phi_L$ by the equation
\begin{equation}\label{eq:rel_orth_modular_forms_theta_int}
  -4\log|\Psi_L(f)(z,h)|=  \Phi_L(f)(z,h) + c(0,0)(2\log|\im(z)|+ \log(2\pi) + \Gamma'(1))
\end{equation}
Equivalently we may write  
\begin{equation}
   -2\log\|\Psi_L(f)(z,h)\|^2_{\text{Pet},c(0,0)/2} = \Phi_L(f)(z,h) + c(0,0)(\log(2\pi) + \Gamma'(1)),
\end{equation}
where $\displaystyle\|\Psi_L(f)(z,h)\|_{\text{Pet},r} = |\Psi_L(f)(z,r)||\im(z)|^r$ is the Petersson norm weight $r$ with $\displaystyle |\im(z)|=|(\im(z), \im(z))|^{1/2}$. 
  \end{enumerate}
\end{theorem}
Note that we get for $h=1$ the original regularized theta lift $\Phi_L(f)(z,1)$ and the original Borcherds lift $\Psi_L(f)(z,1)$ back. Also, in view of Remark \ref{rem:adelic_theta_lift}, we have $\Psi_L(f)(z,h) = \Psi_{L^h}(f^h)(z)$, i. e. the classical Borcherds lift of $f^h$ for the orthogonal group attached to the lattice $L^h$. 

We now address the main question of this paper. 
We proceed as in \cite{BF}, Sect. 1.1. The proof given therein is essentially based on \cite{BF1}, Thm. 6.1, and \cite{Br1}, Thm. 4.23. We present for both theorems a version in our adelic setup fitting to the statement of the Kudla-Millson theta lift and its injectivity in Section \ref{sec:inj_kudla_millson}. As both theorems rely on the fact that the underlying Hermitian space is a Riemannian manifold, we adopt this assumption (although the theorem below should also hold for the more general situation where $X_K$ is a complex orbifold).    

\begin{theorem}\label{thm:rel_theta_lifts}
  Let $V$ be of type $(p,2)$ and $f\in H_{\ell,A^-}^+$ with the Fourier expansion as in \eqref{eq:fourier_exp_xi_k} and suppose that $\alpha(\Gamma_K)=\Gamma(L)$. Then for all $(z,h)\in (D\bs Z(f))\times H(\A_f)$ the identity
  \begin{equation}\label{eq:rel_theta_lifts}
    dd^c\Phi_L(f)(z,h) = \Lambda(\xi_\ell(f))(z,h) + c^+(0,0)\Omega
  \end{equation}
holds,  where $\displaystyle \Omega$ is the K\"ahler form on $D$ as specified in \cite{BF1}.
\end{theorem}
\begin{proof}
  Let $\displaystyle X_K \cong \bigsqcup_i \Gamma_{i}\bs D^+$
  with $\Gamma_i = H(\Q)^+\cap h_iKh_i^{-1} $ and $h_i\in H(\A_f)$ (cf. \eqref{eq:shimura_variety_decomp}).  By virtue of Remark \ref{rem:adelic_theta_lift} for each $h_i$ we have that  $f^{h_i}\in H_{\ell, (A^{h_i})^{-}}$. According to \cite{B}, 13.3 the divisor of $\Psi_{L^{h_i}}(f^{h_i})$ is given by
  \[
  Z(f^{h_i}) = \frac{1}{2}\sum_{\mu\in A}\sum_{n>0}c^+(-n,\mu)Z_i(n,\varphi_{\mu_i},K)
  \]
  with $Z_i(n,\varphi_{\mu_i},K)$ being defined in \eqref{eq:div_components}. 
  Theorem 6.1 in \cite{BF1} combined with Remark \ref{rem:adelic_theta_lift} now yields
  \begin{equation}\label{eq:rel_theta_lifts_components}
  dd^c\Phi_L(f^{h_i})(z) = \Lambda(\xi_\ell(f^{h_i}))(z) + c^+(0,0)\Omega
  \end{equation}
  for all $z\in D\bs Z_i(f^{h_i})$. 
  Now the function $f$ corresponds via \eqref{eq:differential_forms_x_k} to the collection of functions $(f(\cdot,h_i))_i$. As $\Phi_L(f)$ and $\Lambda(\xi_\ell(f))$ descend to differential forms on $X_K$, they are invariant under the left-action of $H(\Q)$ and the right-action of $K$. Therefore, the equations \eqref{eq:rel_theta_lifts_components} lift to the corresponding equation \eqref{eq:rel_theta_lifts} on $X_K$. This equation holds for all $(z,h)\in X_K\bs \sum_i Z(f^{h_i})$, where
  $\sum_i Z(f^{h_i})$ is nothing else than the divisor $Z(f)$ in \eqref{eq:divisor_weak_maasform} implying the claimed result. 
  \end{proof}

  The following theorem is a generalisation of  Theorem 4.23 in \cite{Br1} to meromorphic modular forms on $D\times H(\A_f)$.  
\begin{theorem}\label{thm:weak_converse}
  Assume that $\alpha(\Gamma_K)$ is the discriminant kernel $\Gamma(L)$. Let  $F:D\times H(\A_f)\rightarrow \C$ be a meromorphic modular form of weight $r$, character $\chi$ and level $K\subset H(\A_f)$ with respect to $H(\Q)$ whose divisor is of the form
  \begin{equation}\label{eq:divisor_of_F}
    \Div(F)= \frac{1}{2}\sum_{\mu\in L'/L}\sum_{n>0}c^+(-n,\mu)Z(n,\mu).
  \end{equation}
  Then there exists a weak Maass form $f\in H_{\ell, L^-}^+$ with principal part $\displaystyle \sum_{\mu\in L'/L}\sum_{n>0}c^+(-n,\mu)e(n\tau)\frake_\mu$ such that
  \begin{equation}\label{eq:converse_weak_maass_form}
    \Phi_L(f)(z,h_i) = -2\log\|F(z,h_i)\|_{\text{Pet},\frac{r}{2}} + c_i,
  \end{equation}
  where $(h_i)_i$ with $h_i\in \H(\A_f)$ is a set of coset representatives of the double coset space in \eqref{eq:shimura_variety_decomp} and $(c_i)_i$ is a set of constants. 
\end{theorem}
\begin{proof}
  Let $(F_{h_i}=F(\cdot,h_i))_i$ be the collection of functions on $D^+$ associated to $F$ (as in \eqref{eq:differential_forms_x_k})  with respect to the decomposition \eqref{eq:shimura_variety_decomp}.  By Remark \ref{rem:action_H_A_f}, i), we have that $\alpha(\Gamma_i) = \Gamma(L^{h_i})$ for all $i$.    Taking this into account, the assumptions on $F$ imply that $F_{h_i}$ (restricted to $D^+$) is a meromorphic modular form for the orthogonal group $\Gamma(L^{h_i})$ in the sense of \cite{Br1}, p. 83. According to \eqref{eq:heegner_div_decomp} and \eqref{eq:div_explicit} the component of $\Div(F)$ on $\Gamma_i\bs D^+$ (identified with $\Gamma(L^{h_i})\bs D^+$) is given by
  \begin{equation}\label{eq:divisor_comp_i}
    \begin{split}
      Z_i(F_{h_i})&= \frac{1}{2}\sum_{\mu\in L'/L}\sum_{n>0}c^+(-n,\mu)Z_i(n,\mu_i). 
   \end{split}
    \end{equation}
 Note that $H(\Q)^+$ acts on $F_{h_i}$ by multiplication with the character $\chi$. Thus, $Z_i(F_{h_i})$ may be interpreted as divisor of $F_{h_i}$ (although $F_{h_i}$ is  not a function on $\Gamma_i\bs D^+$). 
 Each component function $F_{h_i}$ satisfies the assumptions of Theorem 4.23 in \cite{Br1}. It follows from this theorem that the regularized  theta lift 
 \begin{equation}\label{eq:reg_theta_poincare}
   \Phi_L(z,h_i) = -\frac{1}{8}\sum_{\mu_i\in L'/L}\sum_{\substack{n\in \Z+Q(\mu)\\n>0}}c^+(-n,\mu)\Phi_{\mu_i,n}(z)
 \end{equation}

satisfies the equation
\begin{equation}
  \Phi_L(z,h_i) = \log\|F_{h_i}(z)\|_{\text{Pet},\frac{r}{2}} + c_i,
\end{equation}
where $c_i$ is some constant. Here $\Phi_{\mu_i,n}$ denotes the regularized theta lift of the Poincare-Maass form of index $(\mu_i, n)$ (see \cite{Br1}, Def. 1.8).   
\end{proof}



\begin{theorem}\label{thm:surjectivity_borcherds_lift}
  Let $m$ be the even rank of the lattice $L$ satisfying $m>\max(6,3+r_0)$ and $m\equiv 0\bmod{4}$. Moreover, assume that the associated discriminant group $A$ is anisotropic and that $\alpha(\Gamma_K)=\Gamma(L)$. Further, let $F:D\times H(\A_f)\rightarrow \C$ be a meromorphic modular form of weight $r$ and level $K\subset H(\A_f)$ with respect to $H(\Q)$ whose divisor is a linear combination of Heegner divisors $Z(h,n)$ (as in \eqref{eq:divisor_of_F}). Then there exists a weakly holomorphic modular form $f\in M^!_{\ell,L^{-}}$ such that $F$  is up to a constant multiple the Borcherds lift $\Psi_L$ of $f$.    
\end{theorem}
\begin{proof}
  The proof is basically the same as in \cite{BF}, Corollary 1.7. From  Theorem \ref{thm:weak_converse} we know that there is a harmonic weak Maass form $f\in H^+_{\ell,L^{-}}$ with 
  \[
  \Phi_L(f)(z,h_i) = -2\log\|F(z,h_i)\|_{\text{Pet},\frac{r}{2}} + c_i
  \]
  for some constant $c_i$ for all $i$. Applying the exterior derivative $dd^c$ on both sides yields
  \begin{equation}\label{eq:ext_derivative}
  dd^c\Phi_L(f)(z,h_i) = -2dd^c\log\|F(z,h_i)\|_{\text{Pet},\frac{r}{2}} = c^+(0,0)\Omega
  \end{equation}
  for all $i$. 
  On the other hand, by Theorem \ref{thm:rel_theta_lifts}
   the identity \eqref{eq:rel_theta_lifts}
  holds for all $(z,h)\in X_K\bs Z(f)$. Comparing the equations \eqref{eq:rel_theta_lifts} and \eqref{eq:ext_derivative}, we may conclude that $\Lambda(\xi_\ell(f))(z,h) = 0$ for all $(z,h)$. By the injectivity of the Kudla-Millson theta lift, Corollary \ref{cor:injectivity_kudla_millson}, it follows $\xi_\ell(f) = 0$. The exactness of the sequence in \eqref{eq:exact_sequence} then implies $f\in M_{\ell,L^{-}}^!$, giving the desired result.
  %
  %
  %
\end{proof}

As a corollary we obtain
\begin{corollary}\label{cor:surjectivity_borcherds_lift}
  Let $m$ be the even rank of the lattice $L$ satisfying $m>\max(6,3+r_0)$ and $m\equiv 0\bmod{4}$. Moreover, let the  associated discriminant form $A$ be anisotropic and $F:D^+\rightarrow \C$ be a meromorphic modular form of weight $r$ and character $\chi$ (of finite order) for the discriminant kernel $\Gamma(L)$  whose divisor is a linear combination of Heegner divisors. Then there exists a weakly holomorphic modular form $f\in M_{\ell,L^-}^!$ such that $F$ is up to a constant multiple the Borcherds lift $\Psi_L$ of $f$. 
\end{corollary}
\begin{proof}
  In view of \cite{Ni}, Prop. 1.4.1, we may infer that $L$ is maximal. Otherwise an overlattice would lead to an isotropic subgroup of $A$, which is clearly a contradiction considering that $A$ is anisotropic.
  Let $X_K=H(\Q)\bs D\times H(\A_f)/K$ the GSpin Shimura variety as specified in Remark \ref{rem:rel_adelic_classic_orth_modforms}. Then we know from this remark that $X_K \cong \Gamma(L)\bs D^+$. This implies that $H(\A_f)=H(\Q)^+K$ because otherwise $X_K$ could not be connected (see the proof of Lemma 5.13 in \cite{Mi}). In this situation we can identify the modular forms in Definition \ref{def:orth_mod_form_D_A_f} with modular form for the orthogonal group $\Gamma(L)$.  Moreover, we get back the original Kudla-Millson theta lift $\Lambda(f)(z,1)$  and the original Borcherds lift $\Psi_L(f)(z,1)$. The application of Theorem \ref{thm:surjectivity_borcherds_lift} (for $h=1$) then concludes the proof.     
\end{proof}

\section{Non-existence of reflective automorphic forms}\label{sec:non_ex_refl_forms}
In this section we apply the converse theorem Corollary \ref{cor:surjectivity_borcherds_lift} to refine a theorem of Scheithauer in \cite{Sch} on the non-existence of reflective automorphic products. 
We first briefly recall the necessary material to present our results. We use \cite{Di} and \cite{Sch} as main references and keep the setting and notation of Section \ref{sec:surj_borcherds}.

\begin{definition}\label{def:vec_val_symmetric}
  We say that  a weakly holomorphic modular form $f$ for $\rho_A$ is {\it symmetric} if $\sigma(f) = f$ for all $\sigma \in \Aut(A)$. We can associate to any weakly holomorphic modular form $f$ a symmetric vector valued modular form of the same weight by
  \begin{equation}\label{eq:symmetrization}
  f^{\Sym} = \frac{1}{|\Aut(A)|}\sum_{\sigma\in \Aut(A)}\sigma(f).
  \end{equation}
$f^{\Sym}$ is called the symmetrization of $f$. 
  \end{definition}
Based on this definition we explain the concept of a symmetric modular form for some orthogonal group. 
\begin{definition}\label{def:reflective_orth_forms}
     Let $\Gamma_L\le \Gamma(L)$ be a subgroup of finite index of the discriminant kernel. A holomorphic modular form $F$ of weight $r\in \Z$ and character $\chi$ for $\Gamma_L$ (as defined e. g. in \cite{Br1}, p 8.3) is called 
  \begin{enumerate}
  \item[i)]
    {\it symmetric} if it is the Borcherds lift of a symmetric weakly holomorphic modular form in the sense of Definition \ref{def:vec_val_symmetric}.
    \item[ii)]
   {\it reflective} if all its zeros are located on divisors of the form $\lambda^\perp$ where $\lambda$ is a root in $L$.
      \end{enumerate}
  \end{definition}

Reflective modular forms can be obtained by applying the Borcherds lift to a vector valued modular form of a certain shape. To phrase this relation, we introduce some subsets of the discriminant group $A$: For $c\in \Z_{> 0}$ and $x\in \Q$ we put
\begin{equation}
  A_{c,x}=\left\{\mu\in A\; |\; \ord(\mu)=c \text{ and } Q(\mu)=x+\Z\right\}. 
\end{equation}
In terms of this subset we have
\begin{proposition}[\cite{Sch}, Sect. 9, \cite{Di}, Prop. 3.18]\label{prop:vec_val_refl_aut_forms}
  Assume that $L$ has square-free level and let $f$ be a modular form in  $M_{\ell,A}^!$ satisfying:
  \begin{enumerate}
  \item[i)]
    For $\mu\in A_{l,1/l}$ the Fourier expansion of $f_\mu$ is given by
    $f_\mu = c(\mu,1/l) + O(1)$ with $c(\mu,1/l)\in \Z_{>0}$. 
  \item[ii)]
    $f_\mu$ is holomorphic at $\infty$ for all other $\mu\in A$.
  \end{enumerate}
  Then the Borcherds lift $\Psi_L(f)$ is a reflective orthogonal modular form. 
\end{proposition}
Dittman proved the converse of Proposition \ref{prop:vec_val_refl_aut_forms} if $L$ additionally splits a hyperbolic plane $U$ over $\Z$.

\begin{proposition}\label{prop:refl_aut_vec_val_mfs}
  Let $L$ be lattice of square-free level and suppose that $L$ splits a hyperbolic plane $U$. If the Borcherds lift $\Psi_L(f)$ is reflective, then the vector valued modular form $f\in M^!_{\ell,A}$ satisfies:
  \begin{enumerate}
    \item[i)]
    For $\mu\in A_{l,1/l}$ the Fourier expansion of $f_\mu$ is given by
    $f_\mu = c(\mu,1/l) + O(1)$ with $c(\mu,1/l)\in \Z_{>0}$. 
  \item[ii)]
    $f_\mu$ is holomorphic at $\infty$ for all other $\mu\in A$.
    \end{enumerate}
\end{proposition}

At this point it is worth mentioning that in \cite{Sch} Scheithauer refers to the vector valued modular forms in Prop. \ref{prop:vec_val_refl_aut_forms} as {\it reflective modular forms}. Moreover, he calls a modular form $F$ for some orthogonal group $\Gamma_L$ as an {\it reflective automorphic product} if $F$ is the Borcherds lift of some vector valued modular as specified in Prop. \ref{prop:vec_val_refl_aut_forms}. In terms of these definitions an important assertion of \cite{Sch} is given in the following theorem.

\begin{theorem}\label{thm:non_exist_autom_refl_prod}
  The number of automorphic products of singular weight which are symmetric and reflective on lattices of type $(p,2)$ with $p>2$, square-free level and $q$-ranks at most $p+1$ is finite. 
\end{theorem}

By means of the converse theorem of this paper, we are able to extend this to result to reflective modular forms  as given in Definition \ref{def:reflective_orth_forms} without assuming that the underlying lattice $L$ splits a sublattice of the form $U\oplus U(N)$ over $\Z$.  Nevertheless, it is necessary that the conditions of the converse theorem in Corollary \ref{cor:surjectivity_borcherds_lift} are satisfied.

We first note that the assumptions on the lattice $L$ in this paper are sufficient to prove that $L$ splits a hyperbolic plane $U$ over $\Z$. Therefore, we may employ Prop. \ref{prop:refl_aut_vec_val_mfs} for subsequent arguments.

\begin{lemma}\label{prop.split_hyperbol_plane}
  Let $L$ be non-degenerate even lattice of type $(p,2)$ with $p\ge 3$ whose associated discriminant group $A=L'/L$ is anisotropic. Then $L$ splits a hyperbolic plane over $\Z$.
\end{lemma}
\begin{proof}
  We first note that the level $N$ of $L$ is square-free as $A$ is anisotropic and $N$ is odd (see e. g. \cite{BEF}, p. Lemma 4.9).
  In view of the type of $L$ we clearly have an isotropic vector $z\in L$, which we can choose to be primitive. It is well known (arguing the same way as in the proof of Lemma 4.6 in \cite{Eb}) that there is a vector $z'\in L'$ associated to $z$ with  $(z,z')=1$.
  Dittmann shows in Prop. 2.73 that $z'$ can be chosen to be isotropic and to satisfy $lz'\in L$ for some positive interger $l$ dividing $N$. In particular, we have a decomposition of the form $L\cong K\oplus U(l)$, where $K = L\cap z^\perp\cap w^\perp$ with $w=lz'$.
It easily seen that $U(l)'/U(l)\cong (\Z/l\Z)^2$ and that this subgroup of $A$ contains isotropic elements (c. f. \cite{Di}, Prop. 1.71). But this is not possible as $A$ is anisotropic. Hence, $l$ must be equal to 1. In this case we have $L\cong K\oplus U$, giving the desired assertion. 
\end{proof}

We are now in a position to state and proof the main result of this section.

\begin{theorem}\label{thm:non_ex_reflective_mfs}
  There are only finitely many lattices (up to isomorphism) that satisfy all conditions of Corollary \ref{cor:surjectivity_borcherds_lift} and that admit a reflective modular form of singular weight for $\Gamma(L)$. 
  \end{theorem}
\begin{proof}
  Let $F$ be a reflective modular form of  weight $\frac{p}{2}-1$  and character of finite order with respect to the discriminant kernel $\Gamma(L)$. 
  By Corollary \ref{cor:surjectivity_borcherds_lift}, there is a weakly modular form $f\in M^!_{\ell,L^-}$ such that $F$ is a constant multiple of $\Psi_L(f)$. As $F$ has singular weight, the Fourier coefficient $c(0,0)$ is equal to $p-2$.  
 By Lemma \ref{prop.split_hyperbol_plane}, we know that $L$ splits a hyperbolic plane $U$. Prop. \ref{prop:refl_aut_vec_val_mfs} then implies that $f$ is a reflective vector valued modular form. The symmetrization of $f$ is a symmetric and reflective (see (3.2) in \cite{Di}). It is easily checked that the constant coefficient of $(f^{\Sym})_0$ is also equal to $c(0,0)=p-2$. Finally, by means of \eqref{eq:anisotropic_finite_modules} it clear that the $q$-rank of each $q$ component of $A$ is smaller than $p+1$. Thus, $f^{\Sym}$ meets all conditions of Thm. 11.1 in \cite{Sch}. Then arguing the same way as in \cite{Sch}, Sect. 11 and 12,  yields that the rank and the level of $L$ is bounded, which gives the claimed assertion.  
\end{proof}

\end{document}